\documentclass[12pt]{amsart}
\usepackage{amssymb}
\usepackage{color}
\usepackage{titletoc}

\usepackage{url,titletoc,enumerate}

\usepackage[usenames,dvipsnames]{xcolor}
\usepackage[pagebackref,colorlinks,linkcolor=BrickRed,citecolor=OliveGreen,urlcolor=blue,hypertexnames=true]{hyperref}

  \def\ol{\overline}
    \def\la{\lambda}
    
    \usepackage{graphicx}

  \usepackage{mathrsfs,accents}

\numberwithin{equation}{section}

\newtheorem{thm}{Theorem}[section]
\newtheorem{theorem}[thm]{Theorem}

\newtheorem{cor}[thm]{Corollary}
\newtheorem{lemma}[thm]{Lemma}
\newtheorem{lem}[thm]{Lemma}
\newtheorem{prop}[thm]{Proposition}
\newtheorem{proposition}[thm]{Proposition}

\theoremstyle{definition}

\newtheorem{conj}         [thm]{Conjecture}

\newtheorem{remark}[thm]{Remark}
\newtheorem{remx}[thm]{Remark}
\newenvironment{rem}
  {\pushQED{\qed}\remx}
  {\popQED\endremx}
\newtheorem{examplex}[thm]{Example}
\newenvironment{example}
  {\pushQED{\qed}\examplex}
  {\popQED\endexamplex}

  \newcommand{\vertiii}[1]{{\left\vert\kern-0.25ex\left\vert\kern-0.25ex\left\vert #1 
    \right\vert\kern-0.25ex\right\vert\kern-0.25ex\right\vert}}

\def\pt{\partial}

\def\mco{\textrm{\rm co$^{\rm mat}$}}
\def\cmco{\overline{\textrm{\rm co}}^{\rm mat}}

\def\bTV{ bent TV screen}

\def\ben{\begin{enumerate}}
\def\een{\end{enumerate}}

\def\blem{\begin{lemma}}
\def\elem{\end{lemma}}

\def\bs{\bigskip}

\def\bem{\begin{pmatrix}}
\def\eem{\end{pmatrix}}

\def\beq{\begin{equation}}
\def\eeq{\end{equation}}

\renewcommand{\subset}{\subseteq}
\renewcommand{\supset}{\supseteq}
\renewcommand{\emptyset}{\varnothing}

\newcommand{\ax}{\langle x\rangle}

\DeclareMathOperator{\Span}{span}

\DeclareMathOperator{\thull}{thull}
\DeclareMathOperator{\cthull}{cthull}

\DeclareMathOperator{\proj}{proj}

\DeclareMathOperator{\inter}{int}
\DeclareMathOperator{\tr}{tr}

\newcommand{\df}[1]{{\bf{#1}}{\index{#1}}}

\def\cC{\mathcal C}

\def\cD{\mathcal D}

\def\RR{\mathbb R}

\def\ga{\gamma}

\def\epsilon{\varepsilon}
\def\eps{\epsilon}
\def\bep{\proof}
\def\eep{\qed}
\def\bec{\begin{conj}}
\def\eec{\end{conj}}
\def\bex{\begin{example}}
\def\eex{\end{example}}

\def\smatmg{\mathbb S_m^g}

\def\smatg{\mathbb S^g}
\def\smath{\mathbb S^h}
\def\smatn{\mathbb S_n}

\def\smatng{\mathbb S_n^g}
\def\smatnh{\mathbb S_n^h}
\def\smatmg{\mathbb S_m^g}
\def\smatdg{\mathbb S_d^g}

\def\cS{\mathcal S}

\def\cN{\mathcal N}

\def\RR{ {\mathbb{R}} }
\def\R{ {\mathbb{R}} }
\def\C{ {\mathbb{C}} }

\def\N{ {\mathbb{N}} }

\def\La{\Lambda}

\def\pt{\partial}
\def\la{\lambda}

\def\pt{\partial}
\def\epsilon{\varepsilon}
\def\eps{\epsilon}

\def\cL{\mathcal L}
\def\Om{\Omega}

\def\mL{\mathfrak{L}}

\def\La{\Lambda}
\def\mbS{\mathbb S}

\makeatletter
\newcommand\meaningbody[1]{%
  {\ttfamily
    \expandafter\strip@prefix\meaning#1}%
}
\makeatother

\makeatletter
\newcommand*{\mydef}[2]{%
    \def#1{#2}%
    {\ttfamily\string#1=\expandafter\strip@prefix\meaning#1}%
}
\makeatother

\def\thull{\mbox{\rm thull}}
\def\cthull{\mbox{\rm cthull}}
\def\cthull{\mbox{\rm cthull}}
\def\tpd{\triangleright}

\def\cT{{\mathcal T}}

\def\fH{{\mathfrak H}}

\def\cY{\mathcal Y}

\def\cTn{\mathcal T_n}

\def\hh{T}
\def\hhs{\mathfrak{T}}

\def\fT{\mathfrak{T}}

\def\hK{\hat{\calK}}

\def\thull{\mbox{\rm thull}}
\def\cthull{\mbox{\rm cthull}}
\def\tpd{\triangleright}

\def\cS{\mathcal S}
\def\cU{\mathcal U}

\def\idem{in situ }
\def\Idem{In Situ }
\def\contra{ex situ }
\def\Contra{Ex Situ }

\def\fHT{\fH^{\mbox{\tiny\rm opp}}}

\title[Projections of LMI Domains and the Tracial Hahn-Banach Theorem]
{The Tracial Hahn-Banach Theorem, Polar Duals,\\[.1cm] Matrix Convex Sets, and\\[.1cm] Projections of Free Spectrahedra}

\author[J.W. Helton]{J. William Helton${}^1$}
\address{J. William Helton, Department of Mathematics\\
  University of California \\
  San Diego}
\email{helton@math.ucsd.edu}
\thanks{${}^1$Research supported by the National Science Foundation (NSF) grant
DMS 1201498, and the Ford Motor Co.}

\author[I. Klep]{Igor Klep${}^{2}$}
\address{Igor Klep, Department of Mathematics, 
The University of Auckland, New Zealand}
\email{igor.klep@auckland.ac.nz}
\thanks{${}^2$Supported by the Marsden Fund Council of the Royal Society of New Zealand. Partially supported by the Slovenian Research Agency grants P1-0222, L1-4292 and L1-6722. Part of this research was done while the author was on leave from the University of Maribor.}

\author[S. McCullough]{Scott McCullough${}^3$}
\address{Scott McCullough, Department of Mathematics\\
  University of Florida\\ Gainesville 
   }
   \email{sam@math.ufl.edu}
\thanks{${}^3$Research supported by the NSF grants DMS 1101137 and 1361501}

\subjclass[2010]{Primary 14P10, 47L25, 90C22; Secondary 13J30, 46L07}
\date{\today}
\keywords{linear matrix inequality (LMI), polar dual, LMI domain, spectrahedron,
spectrahedrop, convex hull, free real algebraic geometry, noncommutative polynomial, cp interpolation, quantum channel, tracial hull, tracial Hahn-Banach Theorem}

\makeindex

\begin{document}

\linespread{1.1}

\begin{abstract}
This article investigates matrix convex sets and introduces their tracial analogs which we call contractively tracial convex sets. In both contexts completely positive (cp) maps play a central role:  unital cp  maps in the case of matrix convex sets and  trace preserving  cp (CPTP) maps in the case of  contractively tracial convex sets.  CPTP maps, also known as  quantum channels,  are fundamental objects in  quantum information theory.

Free convexity is   intimately connected with Linear Matrix Inequalities (LMIs)
 $L(x)=A_0+A_1x_1+\cdots+A_gx_g \succeq0$
and  their matrix convex solution sets $\{ X: \ L(X) \succeq 0 \},$ called \emph{free spectrahedra}. 
The Effros-Winkler Hahn-Banach Separation Theorem for matrix convex sets states that matrix convex sets are solution sets 
of LMIs with operator coefficients.  Motivated in part by cp interpolation problems, we develop the foundations of convex analysis and duality
in the tracial setting, including tracial analogs of the Effros-Winkler Theorem. 

The projection of a free spectrahedron in $ g+h$ variables to $g$ variables 
is a matrix convex set called a \emph{free spectrahedrop}.  As a class, free spectrahedrops
are more general than free spectrahedra, but at the same time more tractable than
general matrix convex sets. Moreover, many matrix convex sets can be approximated from above
by free spectrahedrops.  Here a number of fundamental results for spectrahedrops  and their polar duals are 
established.  For example,
the free polar dual of a free spectrahedrop is again a free spectrahedrop.
We also give a Positivstellensatz for free polynomials that are positive on a free spectrahedrop.
\end{abstract}

\maketitle

\setcounter{tocdepth}{3}
\contentsmargin{2.55em} 
\dottedcontents{section}[3.8em]{}{2.3em}{.4pc} 
\dottedcontents{subsection}[6.1em]{}{3.2em}{.4pc}
\dottedcontents{subsubsection}[8.4em]{}{4.1em}{.4pc}

\linespread{1.18}

\section{Introduction}
This article investigates matrix convex sets from the perspective  of the emerging areas of free real algebraic
geometry and free analysis \cite{Voi04,Voi10,KVV+,MS11,Pope10,AM+,BB07,dOHMP09,HKMfrg,PNA10}.  It also introduces  
 \df{contractively tracial convex sets},  the tracial analogs  of matrix convex sets appropriate
  for the  quantum channel and quantum operation interpolation problems.
   Matrix convex sets arise naturally in a number of contexts, including 
engineering systems theory, 
 operator spaces, systems and algebras and  
 are inextricably linked to unital completely positive (ucp) maps
 \cite{SIG97,Arv72,Pau02,Far00,HMPV09}. On the other hand, 
 completely positive trace preserving  (CPTP) maps
  are central to quantum information theory \cite{NC10,JKPP11}. 
 Hence there is an inherent similarity between matrix convex sets
and structures naturally occurring in  quantum information theory.

  Given positive integers $g$ and $n$, let  \df{$\smatng$}  
 denote the set of $g$-tuples $X=(X_1,\dots,X_g)$ of complex $n\times n$ hermitian matrices
  and let \df{$\smatg$} denote the sequence $(\smatng)_n$. 
We use $M_n$ to denote the algebra of $n\times n$ complex matrices.
  A subset $\Gamma\subset \smatg$ is a sequence
  $\Gamma=(\Gamma(n))_n$ such that $\Gamma(n)\subset \smatng$ for each $n$.
  A \df{matrix convex set} is a subset $\Gamma\subset\smatg$ 
  that is \df{closed with respect to direct sums} and (simultaneous) {\bf conjugation by isometries}.
  \index{closed with respect to conjugation by isometries} 
  Closed under direct sums
  means  if $X\in \Gamma(n)$ and 
  $Y\in \Gamma(m)$, then
\begin{equation}
 \label{eq:dirsums}
 X\oplus Y:= \big( \begin{pmatrix} X_1 & 0 \\ 0 & Y_1\end{pmatrix}, \dots, \begin{pmatrix} X_g & 0 \\ 0 & Y_g \end{pmatrix} \big) \in \Gamma(n+m).
\end{equation}
  Likewise, closed under conjugation by isometries means if $X\in\Gamma(n)$ and $V$ is an $n\times m$ isometry, then
\[
 V^* X V := (V^* X_1 V, \dots, V^* X_g V)\in\Gamma(m).
\]
The simplest examples of matrix convex sets arise as solution sets of linear matrix inequalities (LMIs). 
   The use of LMIs is a major advance in systems engineering in the past two decades \cite{SIG97}. Furthermore, 
   LMIs underlie the theory of semidefinite programming \cite{BPR13,BN02}, itself a recent major innovation in convex optimization \cite{Ne06}.

  Matrix convex sets determined by LMIs are based on a free analog of an affine linear functional,
 often called a \df{linear pencil}. \index{$L_A$}  Given a positive integer $d$ and $g$ hermitian $d\times d$ matrices  $A_j,$ let
\begin{equation}
 \label{eq:LMIx}
 L(x) = A_0 + \sum_{j=1}^g  A_j x_j.
\end{equation}
This linear pencil is often denoted by $L_A$ to emphasize the dependence on $A$.
  In the case that $A_0=I_d$, we call $L$ {\bf monic}.  \index{monic linear pencil}
  Replacing $x\in\mathbb R^g$ with a tuple $X=(X_1,\dots,X_g)$ of $n\times n$ hermitian matrices
 and letting $W\otimes Z$ denote the Kronecker product
  of matrices leads to  the evaluation of the 
 free affine linear functional,
\begin{equation}
 \label{eq:LMI}
 L(X) = A_0 \otimes I_n +\sum A_j\otimes X_j.
\end{equation}
 The inequality $L(X)\succeq 0$
 is a \df{free linear matrix inequality (free LMI)}. The solution set $\Gamma$ of 
 this free LMI is the sequence of sets
\[
 \Gamma(n) =\{X\in\smatng : L(X)\succeq 0\}
\]
  and is known as a \df{free spectrahedron} (or \df{free LMI domain}). 
 It is easy to see that  $\Gamma$ is  a matrix convex set.

 By the Effros-Winkler matricial Hahn-Banach Separation Theorem \cite{EW97}, (up to a technical hypothesis)
 every matrix convex set is the solution set of $L(X)\succeq 0,$ as in equation \eqref{eq:LMI},
  of some monic linear pencil {\it provided}
  the $A_j$ are allowed to be hermitian operators on a (common) Hilbert space. More precisely, every matrix
 convex set is a (perhaps infinite) intersection of free spectrahedra.  Thus, being a spectrahedron 
  imposes a strict finiteness condition on a matrix convex set.  

   In between (closed) matrix convex sets and spectrahedra lie the class of domains we call \df{spectrahedrops}.
Namely coordinate projections   of free spectrahedra. A subset  $\Delta \subset \smatg$ is a \df{free spectrahedrop} if there exists a pencil
\[
  L(x,y) = A_0 + \sum_{j=1}^g  A_j x_j  + \sum_{k=1}^h  B_k y_k,
\]
 in $g+h$ variables such that  
\beq\label{eq:deltaDrop}
 \Delta(n) =\{X\in\smatng: \exists Y \in\smatnh  \mbox{ such that } L(X,Y)\succeq 0\}.
\eeq

  In applications, presented with a convex set, one  would like, 
  for optimization purposes say, to know if it is a spectrahedron or 
  a spectrahedrop.  Alternately,  presented with an algebraically defined set $\Gamma \subset \smatg$ that is not
  necessarily convex,   it is natural to consider the relaxation obtained by replacing $\Gamma$ 
  with its matrix convex hull or an approximation thereof. Thus, it is of interest to know when
  the convex hull of a set is a spectrahedron or perhaps a spectrahedrop. 
  An approach to these problems via approximating from above by spectrahedrops was 
  pursued in the article \cite{Lasse}. Here we develop the duality approach. Typically
  the second polar dual of a set is its closed matrix convex hull.

\subsection{Results on Polar Duals and Free Spectrahedrops}
\label{sec:polardualresults}
 We list here our main results on free spectrahedrops and polar duals.  For the reader 
  unfamiliar with the terminology, the 
  definitions not already introduced can be found in Section \ref{sec:2}
  with the exception of {\it free polar  dual} whose definition 
  appears in Subsection \ref{sec:free polar dual basics}.

\ben
\item
A perfect  free Positivstellensatz (Theorem \ref{thm:posDrop}) for any symmetric free polynomial $p$ on a free spectrahedrop $\Delta$
as in \eqref{eq:deltaDrop}.
It says that $p(X)$ is positive semidefinite for all $X \in \Delta$
if and only if $p$ has the form
$$ p(x)=  f(x)^* f(x)  + \sum_\ell q_\ell(x)^* L(x,y) q_\ell (x)
$$
where $f$ and and $q_\ell$ are vectors with polynomial entries.
If the degree of $p$ is less than or equal to $2r +1$,
then  $f$ and $q_\ell$  have degree no greater than $r$;
\item
The free polar dual of a free spectrahedrop is a free spectrahedrop (Theorem \ref{thm:polarPolar} 
 and Corollary \ref{cor:dual of drop is drop});
\item
The matrix convex hull of a union of finitely many bounded free spectrahedrops is a
bounded free spectrahedrop (Proposition \ref{prop:hull of union});
\item
 A matrix convex set is, in a canonical sense, generated
 by a finite set (equivalently a single point) if and only if 
 it is the polar dual of a free spectrahedron (Theorem \ref{thm:polarLMI}).
\een

\subsection{Results on Interpolation of cp Maps and Quantum Channels}

A completely positive ({\bf cp}) map $M_n\to M_m$ that  is trace preserving is
called a {\bf quantum channel}, and a cp map
that  is trace non-increasing for positive semidefinite
arguments is a {\bf quantum operation}.
These maps figure prominently  in quantum information theory
 \cite{NC10}. 

The cp interpolation problem is formulated as follows.
Given $A \in \mbS_n^g$ and  $B\in\mbS_m^g$,
does  there exists a cp map $\Phi:M_n\to M_m$ such that, for $1\le \ell\le g$,
\[
 \Phi(A_\ell) = B_\ell ?
\]
One can require further  that $\Phi$ be unital, a quantum channel or a quantum operation.
Imposing either of the latter two constraints pertains to
quantum information theory  \cite{Hardy, Klesse, NCSB}, where one is interested
in quantum channels (resp., quantum operations) 
that send a prescribed  set of
quantum states into another set of quantum states.

\subsubsection{Algorithmic Aspects}
A  byproduct of the methods used in this paper
and in \cite{HKM+} produces 
 solutions to these cp  interpolation problems in the form of an  algorithm,
 Theorem \ref{thm:interp} in Subsection
 \ref{subsec:interp}.  Ambrozie and Gheondea \cite{AGprept} 
solved these interpolation problems with LMI algorithms.
While  equivalent to theirs, 
  our solutions are formulated as concrete LMIs that can be solved with a standard semidefinite programming (SDP) solver. \index{SDP}
These interpolation results are a basis for proofs of results outlined in Section \ref{sec:polardualresults}. 

\subsection{Free Tracial Hahn-Banach Theorem}
Matrix convex sets are 
 closely connected with ranges of unital cp maps.
  Indeed, given a tuple $A\in\smatmg$, the matrix convex hull of  
   the set $\{A\}$ is the sequence of sets
\[
  \big( \{B\in \smatng: B_j = \Phi(A_j) \mbox{ for some ucp map } \Phi:M_m\to M_n \} \big)_n.
\]
 From the point of view of 
  quantum information theory it is natural to consider
hulls of ranges of quantum operations. 
We say that $\cY\subseteq \mbS^g$ is \df{contractively tracial} if for all positive integers $m,n$, elements 
 $Y\in \cY(m)$, and  finite collections  $\{C_\ell\}$ of $n\times m$ matrices such that
\[
 \sum C_\ell^\ast C_\ell \preceq I_m, 
\]
 it follows that $\sum C_j Y C_j^\ast \in \cY(n)$.  
 It is clear that an intersections of contractively tracial sets 
  is again contractively tracial, giving rise, in the usual way,  to 
  the notion of the  {\bf contractive tracial hull}, denoted \df{cthull}.
 For a tuple $ A$, 
\[
 \cthull(A)= \{B:  \Phi(A)=B \ \ \mbox{for some quantum operation } \Phi \}.
\]
  While the unital and quantum interpolation
  problems have very similar formulations, contractive  tracial hulls possess far
  less structure than matrix convex hulls. A subset $\mathscr Y\subset \mbS^g$ is \df{levelwise convex} if each $\mathscr Y(m)$ is convex (as a subset of $\smatmg$). (Generally \df{levelwise} refers to a property holding for each $\mathscr Y(m)\subseteq \smatmg$.) As is easily seen,
contractive tracial hulls
  need not be levelwise convex nor closed with respect to direct sums.  
However they do have a few good properties. These  we develop in Section \ref{sec:q1}.

 Section \ref {sec:tHBan}  contains
 notions of free spectrahedra and corresponding
 Hahn-Banach type separation theorems tailored to the tracial setting. 
 To understand {\em convex} contractively tracial sets,
 given $B\in\mbS_k^g$, let $\fH_B = (\fH_B(m))_m$ denote the sequence of sets
\[
 \fH_B(m)  = \big\{ Y\in\smatmg: \exists T\succeq 0, \ \tr(T)\le 1, \ \  I\otimes T- \sum B_j\otimes Y_j \succeq 0\big\}.
\]  
 We call $\fH_B$ a \df{tracial spectrahedron}.  
(Note that $\fH_B$ is not closed under direct sums, and thus it is not a matrix convex set.)
These $\fH_B$ are all contractively tracial and levelwise convex.
Indeed
for such structural reasons, and in view of the tracial Hahn-Banach separation theorem 
 immediately below, 
 we believe these to be the natural analogs of free spectrahedra in the tracial context.

\begin{theorem}[cf.~Theorem \ref{thm:tracehull}]
 \label{thm:tracehullIntro} 
   If $\cY\subset\smatg$ is contractively tracial,  levelwise convex and closed, and 
   if $Z \in \mathbb S_m^g$ is not in  $\mathcal Y(m)$, then there exists a $B\in \mathbb S_m^g$ such that
   $\mathcal Y\subseteq \fH_B$, but $Z\notin\fH_B$.
\end{theorem}

 Because of the asymmetry between $B$ and $Y$ in the definition of $\fH_B$, there is a second type of
  tracial spectrahedron.  Given $Y\in\mbS_k^g$, we define the \df{opp-tracial spectrahedron} 
  as the sequence $\fHT_Y = (\fHT_Y(m))_m$ 
\begin{equation}
 \label{eq:fHT}
 \fHT_Y(m)=\{B\in\smatmg:  \exists T\succeq 0, \ \tr(T)\le 1, \ \  I\otimes T- \sum B_j\otimes Y_j \succeq 0\big\}.
\end{equation}
 Proposition \ref{prop:doubleup} computes the hulls resulting from the two different double duals
 determined by the two notions of tracial spectrahedron.

\subsection{Reader's guide}
The paper is organized as follows. Section \ref{sec:2} introduces terminology and notation used throughout the paper.
Section \ref{sec:cp3} solves the cp interpolation problems, 
and includes a background section on cp maps. 
Section \ref{sec:pdual4} contains our main results on polar duals,
free spectrahedra and 
free spectrahedrops. It uses the results of Section \ref{sec:cp3}.  In particular, 
we show that a matrix convex set is finitely generated
if and only if  it is the polar dual of a free spectrahedron (Theorem \ref{thm:polarLMI}). 
 Furthermore, we prove that 
 the polar dual of a free spectrahedrop is again a free spectrahedrop (Theorem \ref{thm:polarPolar}).
Section \ref{sec:PosSS} contains the ``perfect''
Convex Positivstellensatz for free polynomials positive semidefinite on free spectrahedrops.
 The proof depends upon the results of Section \ref{sec:pdual4}.
 In Section \ref{sec:q1} we  introduce tracial sets and hulls and discuss their connections with the quantum interpolation problems from Section \ref{sec:cp3}. Finally, Section \ref{sec:tHBan} introduces tracial spectrahedra and proves a Hahn-Banach separation theorem in the tracial context, see Theorem \ref{thm:tracehull}. This theorem is then used to 
 suggest corresponding notions of duality. Section \ref{sec:exs} contains examples.
 
\subsection*{Acknowledgments}
The authors thank Man-Duen Choi for stimulating discussions and the referees for their valuable suggestions for improving the exposition.

\section{Preliminaries}\label{sec:2}

This section introduces terminology and
presents preliminaries on free polynomials, free
sets and free convexity needed in the sequel.

\subsection{Free Sets} 
  A set $\Gamma \subset \smatg$   is closed with respect to 
  {\bf (simultaneous) unitary  conjugation} \index{unitary conjugation, simultaneous} if for each $n,$ each $A\in \Gamma(n)$
 and each $n\times n$ unitary matrix $U$,
\[
 U^* A U = (U^* A_1U,\dots, U^* A_gU)\in \Gamma(n). 
\]
 The set $\Gamma$ is a  \df{free set} if it is closed with respect to direct sums (see equation \eqref{eq:dirsums})  and 
simultaneous  unitary conjugation. In particular, a matrix convex set is a free set.
We refer the reader to 
 \cite{Voi04,Voi10,KVV+,MS11,Pope10,AM+,BB07}
for a systematic study of free sets and free function theory.
The set $\Gamma$ is
 {\bf (uniformly) bounded}\index{uniformly bounded} \index{bounded, uniformly} if there is a $C\in\R_{>0}$ such that $C-\sum X_j^2 \succeq 0$
  for all $X\in\Gamma$.

\subsection{Free Polynomials}
 One natural way free sets arise is as the nonnegativity set of a free polynomial.
  Given a positive integers $\ell$ and $\nu$,   let $\C^{\ell\times \nu}$ denote the collection of $\ell\times \nu$  matrices.
An expression of the form
 \[
  P=\sum_{w} B_w w\in\C^{\ell\times\nu}\ax,
\]
  where $B_w\in \C^{\ell\times\nu}$, and the
 sum is a finite sum over the words in the variables $x$, is a \df{free (noncommutative) matrix-valued polynomial}.  \index{free polynomial}
 The collection of all $\ell\times \nu$-valued free polynomials is denoted
  $\C^{\ell\times\nu}\ax$ and $\C\ax$ denotes the set of scalar-valued free polynomials. We use $\C^{\ell\times\nu}\ax_k$ to denote
  free polynomials of degree $\leq k$. Here the degree of a word is its \df{length}.   The free polynomial $P$ is evaluated at an $X\in\smatng$ by
\[
  P(X) =  \sum_{w\in\ax} B_w \otimes w(X) \in \C^{\ell n\times \mu n},
\]
where $\otimes$ denotes the (Kronecker) tensor product.

 There is a natural \df{involution} ${}^*$ on words that reverses the order. This
 involution extends to $\C^{\ell\times\nu}\ax$ by
\[
  P^* =\sum_w B_w^*  w^*\in\C^{\mu\times\ell}\ax. 
\]
  If $\mu=\ell$ and $P^*=P$, then  $P$ is \df{symmetric}.
 Note that if $P\in\C^{\ell\times \ell}\ax$ is symmetric, and $X\in\smatn^{g}$, then
 $P(X)\in\C^{\ell n\times \ell n}$ is a hermitian
 matrix.

\subsection{Free Semialgebraic Sets}
 The \df{nonnegativity set} $\cD_P \subset \smatg$ of a symmetric free polynomial is the sequence  of sets 
\[
 \cD_P(n) =\{X\in\smatng: P(X)\succeq 0\}.
\]
  It is readily checked that $\cD_P$ is a free set. By analogy with (commutative) real algebraic geometry,
  we call $\cD_P$ a basic \df{free semialgebraic set}. 
  Often it is assumed that $P(0)\succ 0$. 
  The free set $\cD_P$ has the additional property that it is \df{closed with respect to restriction to reducing subspaces}; that is,
  if $X\in\cD_P(n)$ and $\mathcal H\subset \C^n$ is a reducing (equivalently invariant) subspace for $X$ of dimension $m$, then 
  that $X$ restricted to $\mathcal H$ is in $\cD_P(m)$. 

\subsection{Free Convexity}
In the case that $\Gamma$ is matrix convex, 
it is easy to show that $\Gamma$ is levelwise convex.
 More generally, if $A^\ell=(A^\ell_1,\dots,A^\ell_g)$ are in $\Gamma(n_\ell)$ 
  for $1\le \ell \le k$, then 
  $A=\bigoplus_{\ell=1}^k  A^\ell\in \Gamma(n),$
  where $n=\sum n_\ell$.
 Hence, if $V_\ell$ are $n_\ell \times m$ matrices (for some $m$) such that $V=(V_1^* \, \dots \, V_k^*)^*$   is an isometry (equivalently $\sum_{\ell=1}^k V_\ell^* V_\ell =I_m$), then 
\begin{equation}
\label{eq:ccomb}
 V^* A V = \sum^k_{\ell=1} V_\ell^* A^\ell V_\ell \in \Gamma (m) .
\end{equation}
 A sum as in equation \eqref{eq:ccomb} is a 
 \df{matrix (free) convex combination} of the 
 $g$-tuples $\{A^\ell: \ \ell =1, \dots, k \}$.

\begin{lemma}[\protect{\cite[Lemma 2.3]{Lasse}}]
 \label{lem:capped}
  Suppose  $\Gamma$ is a free subset of $\smatg$. 
\ben[\rm(1)]
 \item
  If  $\Gamma$ is closed with respect to  restriction to reducing subspaces,
  then the following are equivalent: 
 \ben[\rm (i)]
  \item  $\Gamma$ is matrix convex; and
  \item  $\Gamma$ is levelwise convex. 
 \een
\item
  If $\Gamma$ is $($nonempty and$)$ matrix convex, then $0\in \Gamma(1)$
  if and only if  $\Gamma$ is 
    \df{closed with respect to ({\it simultaneous}) conjugation by contractions}. 
\een
\end{lemma}

Convex subsets  of $\R^g$ are defined 
as intersections of half-spaces and are thus described by linear functionals.
Analogously, 
matrix convex subsets of $\smatg$ are defined by linear pencils; cf.~\cite{EW97,HM12}.
We next present basic facts about  linear pencils and their associated matrix convex sets.

\subsubsection{Linear Pencils}
  Recall the definition (see equation \eqref{eq:LMIx}) of the (affine) linear pencil $L_A(x)$ associated to a tuple $A=(A_0,\dots,A_g)\in\mathbb S_k^{g+1}$. \index{$L_A$}
  In the case that $A_0=0$; i.e.,  $A=(A_1,...,A_g) \in \mathbb{S}_k^g$, let \index{$\varLambda_A$}
\[
     \varLambda_A(x)= \sum_{j=1}^g A_j x_j
\]
    denote the corresponding \df{homogeneous (truly) linear pencil}   and  \index{$\mL_A$}
\[
   \mL_A = I- \varLambda_A
\]
  the associated \df{monic linear pencil}.

  The pencil $L_A$ (also $\mL_A$) is a free polynomial with matrix coefficients, so is naturally evaluated on $X \in \smatng$
using (Kronecker's) tensor product yielding equation \eqref{eq:LMI}. 
The free semialgebraic set $\cD_{L_A}$  
is easily seen to be matrix convex.  We will refer to 
 $\cD_{L_A}$ 
 as  a \df{free spectrahedron} or
 \df{free LMI domain}  and
say that a free set $\Gamma$ is \df{freely LMI representable}
 if there is a linear pencil $L$ such that $\Gamma =\cD_{L}$.
 In particular, if $\Gamma$ is freely LMI representable with a monic $\mL_A$, then $0$ is in the interior of $\Gamma(1)$.

The following is a special case  of a theorem due to
Effros and Winkler \cite{EW97}. (See also \cite[Theorem 3.1]{HKM+}.)
  Given a free set $\Gamma$, if $0\in \Gamma(1)$,
  then $0\in\Gamma(n)$  for each $n$.  
 In this case we will write $0\in\Gamma$. 

\begin{theorem}
 \label{prop:sharp}
  If $\cC=(\cC(n))_{n\in\N} \subseteq\smatg$ is a closed matrix convex set containing $0$ and
  $Y\in\smatmg$ is not in $\cC(m)$, then there
 is a monic linear pencil $\mL$ of size $m$ such that
  $\mL(X)\succeq 0$ for all  $X\in\cC$, but
  $\mL(Y)\not\succeq 0$. 
\end{theorem}

By the following 
 result from \cite{HM12}, linear matrix inequalities  account for
  matrix convexity  of free semialgebraic sets.

\def\gP{\mathfrak P}
\begin{theorem}
 \label{thm:HMlmirep}
  Fix $p$ a 
   symmetric matrix polynomial. 
   If 
   $p(0)\succ 0$  and  the strict positivity set 
   $\gP_p=\{X: p(X)\succ0\}$  of $p$ is bounded, 
then   $\gP_p$ 
       is matrix convex if and only if 
   if is freely LMI representable with a monic pencil. 
\end{theorem}

\subsection{A Convex Positivstellensatz and LMI domination}
Positivstellens\"atze are pillars of real algebraic geometry \cite{BCR98}.
We next recall the  Positivstellensatz for a free polynomial $p$. It is the algebraic certificate for  nonnegativity of $p$ on the free spectrahedron $\cD_L$ from \cite{HKM12}.  
It is ``perfect'' in the sense that $p$ is only assumed to be nonnegative
on $\cD_L$, and we obtain  degree bounds on the scale of $\deg(p)/2$
for the polynomials involved in the positivity certificate.
In Section \ref{sec:PosSS}, 
we will extend this Positivstellensatz to free spectrahedrops (i.e., projections of free spectrahedra). See  Theorem \ref{thm:posDrop}.

\begin{thm}\label{thm:convexPos}
Suppose $\mL$ is a monic linear pencil. A matrix polynomial 
 $p$ is positive semidefinite
 on $\cD_{\mL}$ if and only if it
 has a weighted sum of squares representation with optimal degree bounds:
\[ 
 p  = s^* s   + \sum_j^{\rm finite} f_j^* \mL f_j,
\] 
where
 $s, f_j$ are matrix polynomials of degree no greater than
 $\frac{\deg(p) }{2}$.

In particular, if  $\mL_A$, $\mL_B$ are monic linear pencils, then 
$\cD_{\mL_B} \subset \cD_{\mL_A}$ if and only if there exists
a positive integer $\mu$ and a contraction $V$ such that 
\beq\label{eq:lmi-domination2}
 A= 
V^*(I_\mu\otimes B) V. 
\eeq
In the case $\cD_{\mL_B}$ is bounded, $V$ can be chosen to be an isometry.
\end{thm}

\begin{proof}
The first statement is \cite[Theorem 1.1]{HKM12}. Applying this result
to the \df{LMI domination problem}  $\cD_{\mL_B}\subseteq\cD_{\mL_A}$, we see $\cD_{\mL_B}\subseteq \cD_{\mL_A}$  is equivalent to
\begin{equation}
\label{eq:ASB}
\mL_A(x) = S^*S + \sum_{j=1}^\mu V_j^* \mL_B(x) V_j
\end{equation}
for some matrices $S,V_j$;
i.e., 
\begin{align}\label{eq:LMIdominate1}
I & = S^*S + \sum_j V_j^* V_j= S^*S+V^*V \\
 A& =\sum_{j=1}^\mu V_j^* B V_j 
= V^*(I_\mu\otimes B) V,\label{eq:LMIdominate2}
\end{align}
where $V$ is the block column matrix of the $V_j$. 
Equation \eqref{eq:LMIdominate1} simply says that $V$ is a contraction, and \eqref{eq:LMIdominate2} is \eqref{eq:lmi-domination2}.  The last statement is proved in \cite{HKM+}. Alternately, as is shown in \cite[Proposition~4.2]{HKM12}, if $\cD_{\mL_B}$ is bounded, then there are finitely many  matrices $W_j$  such that 
\[
 I = \sum W_j^* \mL_B(x) W_j.
\]
Writing $S^*S = \sum (W_jS)^* \mL_B(x) (W_jS)$ and substituting into equation \eqref{eq:ASB} completes the proof.
\end{proof}

Example \ref{ex:fails} shows that it is not necessarily possible to choose $V$ an isometry in equation \eqref{eq:LMIdominate2} in absence of additional hypothesis on the tuple $B$.

\section{Completely Positive Interpolation}\label{sec:cp3}
Theorem \ref{thm:interp} provides a solution to three cp interpolation problems in terms of concrete LMIs that can be solved with a standard semidefinite programming (SDP) solver. The unital cp interpolation problem comes from efforts to understand matrix convex sets that arise in convex optimization. Its solution  plays an important role in the
 proof of the main result on the polar dual of a free spectrahedrop,
 Theorem \ref{thm:polarPolar}, via its appearance in the proof of
  Proposition \ref{prop:liftPolar}. 

The trace preserving and  trace non-increasing cp interpolation problems 
arise  in quantum information theory and the study of quantum channels, where one is interested in sending one (finite) set of  prescribed  quantum states to another.

\subsection{Basics of Completely Positive Maps}
 This subsection collects basic facts about completely positive (cp) maps
 $\phi:\cS \to M_d$, where $\cS$ is a subspace of $M_n$ closed
 under conjugate transpose (see for instance \cite{Pau02}) containing a positive definite matrix. 
 Thus $\cS$ is a \df{operator system}.

Suppose $\cS$ is a subspace of $ M_n$ closed under conjugate transpose,  $\phi:\cS \to M_d$ is a linear map and $\ell$ is a positive integer.
The ($\ell$-th) \df{ampliation} $\phi_\ell :M_\ell(\cS)\to M_\ell(M_d)$  of $\phi$ is defined by
 by applying $\phi$ entrywise,
\[
 \phi_\ell (S_{j,k}) = \begin{pmatrix} \phi(S_{j,k})\end{pmatrix}.
\]
 The map  $\phi$ is
   \df{symmetric} if $\phi(S^*)=\phi(S)^*$ and it is
  \df{completely positive} if each $\phi_\ell$ is positive in 
  the sense that if $S\in M_\ell(\cS)$ is positive semidefinite, 
  then so is $\phi_\ell(S) \in M_\ell(M_d)$.   In what follows, often
 $\cS$ is a subspace of $\mathbb S_n$ (and is thus automatically closed under 
  the conjugate transpose operation).

 The \df{Choi matrix} of a mapping $\phi: M_n\to M_d$ is the 
 $n\times n$ block matrix with $d\times d$ matrix entries given by
\[
( C_\phi)_{i,j} = \big(\phi(E_{i,j})\big)_{i,j}.
\]
 On the other hand, a matrix $C=(C_{i,j}) \in M_n(M_d)$ determines a mapping
 $\phi_C:M_n\to M_d$ by $\phi_C(E_{i,j}) = C_{i,j}\in M_d$. 
 A matrix $C$ is {\bf a Choi  matrix for} $\phi:\cS\to M_d$,  if
 the mapping $\phi_C$ agrees with $\phi$ on $\cS$.

\begin{theorem}
 \label{thm:cpstuff}
 For $\phi:M_n\to M_d$, the following are equivalent:
 \begin{enumerate}[\rm (a)]
  \item $\phi$ is completely positive;
  \item the Choi matrix $C_\phi$ is positive semidefinite.
  \end{enumerate}

   Suppose $\cS\subseteq M_n$ 
   is an operator system. 
   For a symmetric $\phi:\cS\to M_d,$  the following are equivalent:
 \begin{enumerate}[\rm (i)]
  \item $\phi$ is completely positive;
  \item $\phi_d$ is positive;
  \item there exists a completely positive mapping $\Phi:M_n\to M_d$ extending $\phi$;
  \item there is a positive semidefinite Choi matrix for $\phi$; 
  \item there exists $n\times d$ matrices $V_1,\dots,V_{nd}$ such that
\beq\label{eq:choi}
 \phi(A) = \sum V_j^* A V_j.
\eeq
 \end{enumerate}

  Finally, for a subspace  $\cS$  of $M_n,$ a mapping $\phi:\cS\to M_d$ has
  a completely positive extension $\Phi:M_n\to M_d$ if and only if $\phi$
  has a positive semidefinite Choi matrix. 
\end{theorem}

\def\bel{\begin{lemma} }
\def\eel{\end{lemma} }

\begin{lem}
 The cp mapping $\phi:M_n\to M_d$ as in \eqref{eq:choi} is
 \ben[\rm(a)]
 \item
  unital (that is, $\phi(I_n)=I_d$) if and only if 
 \[
      \sum_{j} V_j^*V_j = I; \]
  \item
trace preserving if and only if 
\[ 
      \sum_{j} V_jV_j^* = I;\] 
\item  
trace non-increasing for positive semidefinite matrices
(i.e., $\tr(\phi(P))\leq\tr(P)$ for all positive semidefinite $P$)
if and only if 
\[ 
      \sum_{j} V_jV_j^* \preceq I. \] 
\een
\end{lem}

\begin{proof}
We prove (c) and leave items (a) and (b) as an easy exercise for the reader.
For $A\in M_n$, 
\[
\tr(\phi(A)) = \sum_j \tr (V_j^* A V_j) = \tr \big( A \sum_j V_jV_j^* \big).
\]
Hence the trace non-increasing property for $\phi$ is equivalent to
\[
\tr\big( P (I-\sum_j V_jV_j^*) \big) \geq 0
\]
for all positive semidefinite $P$, i.e., $I-\sum_j V_jV_j^*\succeq0$.
\end{proof}

\begin{prop} 
 The linear mapping $\phi:M_n\to M_d$ is
 \ben[\rm(a)]
 \item
  unital (that is, $\phi(I_n)=I_d$) if and only if its  Choi matrix $C$ satisfies 
 \[ 
  \sum_{j=1}^n C_{j,j} = I;\] 
  \item
trace preserving if and only if its  Choi matrix $C$ satisfies
\[ 
 (\tr(C_{i,j}))_{i,j=1}^n =I_n; 
\] 
\item 
trace non-increasing for positive semidefinite matrices
(i.e., $\tr(\phi(P))\leq\tr(P)$ for all positive semidefinite $P$)
if and only if 
\[ 
 (\tr(C_{i,j}))_{i,j}  \preceq I_n,
\] 
\een
where $C$ is the Choi matrix for $\phi$. 
\end{prop}

\begin{proof}
Statement (a) follows from
\[
\phi(I_n) =  \phi \big( \sum_{j=1}^n E_{j,j} \big) = \sum_{j=1}^n C_{j,j},
\]
where $C$ is the  Choi matrix for $\phi$. Here $E_{i,j}$ denote the matrix units, 

For (b), let $X=\sum_{i,j=1}^n \alpha_{i,j} E_{i,j}$. Then
\[
\begin{split}
\tr(X) & = \sum_{i=1}^n \alpha_{i,i} \\
\tr(\phi(X)) & = \sum_{i,j=1}^n \alpha_{i,j} \tr( C_{i,j} ).
\end{split}
\]
Since $\tr(\phi(X))=\tr(X)$ for all $X$, this linear system yields
$\tr(C_{i,j})=\delta_{i,j}$ for all $i,j$.

Finally, for statement (c), 
 if $\phi$ is trace non-increasing, choosing $X = x x^*$ a rank one matrix, $X=(x_i x_j),$ we find that
\[
  \sum x_i x_j \tr(C_{i,j})  = \tr(\phi(X)) \le \tr(X) = \sum x_i^2.
\]
Hence $I - ( \tr(C_{i,j}) )  \succeq 0.$
Conversely, if $I -( \tr(C_{i,j})) \succeq 0$, 
 then for any positive semidefinite rank one matrix $X$, the computation above shows that 
  $\tr(\phi(X)) \le \tr(X).$  Finally, use the fact that any 
 positive semidefinite  matrix is a sum of  rank one positive semidefinite matrices to complete the proof.
\end{proof}

The Arveson extension theorem \cite{Arv69} says that  any cp (resp. ucp)  map on an operator system extends to a cp (resp. ucp) map on the full algebra.  Example \ref{ex:no tracial extension} shows that a TPCP map need not extend to a TPCP map on the full algebra.

\subsection{Quantum Interpolation Problems and Semidefinite Programming}
\label{sec:interpAlgor}\label{subsec:cpInterpol}
\label{subsec:interp}

\def\jus{ \noindent {\bf Justification.}}

\def\bs{\bigskip}

The cp interpolation problem is formulated as follows.
Given $A^1 \in \mbS_n^g$ and given $A^2$ in $\mbS_m^g$,
does there exist a cp map $\Phi: M_n \to M_m $ such that 
\[ 
 A^2_\ell= \Phi(A^1_\ell) \quad\text{ for }\quad \ell=1, \ldots, g?
\] 
One can require further  that 
\ben 
\item
$\Phi$ be \df{unital}, or
\item
$\Phi$ be \df{trace preserving}, or
\item
$\Phi$ be \df{trace non-increasing} in the sense that
$\tr(\Phi(P))\leq\tr(P)$ for positive semidefinite $P$.
\een

Our solutions to these interpolation problems  are formulated as concrete LMIs that can be solved with a standard semidefinite programming (SDP) solver.  They are equivalent to, but stated quite differently than, the earlier results in \cite{AGprept}.

\begin{theorem}
\label{thm:interp}
Suppose, for $\ell=1,\ldots,g$ the matrices  $A^1_\ell\in\mbS_n$ and  $A^2_\ell\in\mbS_m$ are symmetric. 
Let $\alpha_{p,q}^\ell$ denote the $(p,q)$ entry of
$A^1_\ell$.

There exists a cp map $\Phi:M_n\to M_m$ that solves the interpolation problem
\[\Phi( A^1_\ell ) = A^2_\ell,
\quad \ell=1, \cdots, g
\] 
if and only if 
the following feasibility semidefinite programming problem has a solution: 
\beq
\label{eq:tausdp}
(C_{p,q})_{p,q=1}^{n} := 
C\succeq0,
 \qquad 
\qquad
\forall \ell=1,\ldots,g: \; 
\sum_{p,q}^{n}  \alpha_{p,q}^\ell
C_{p,q}=  A^2_\ell,
\eeq
 for the unknown $mn\times mn$ symmetric matrix $C=(C_{p,q})_{p,q=1}^n$ consisting of $m\times m$ blocks $C_{p,q}$. 
Furthermore,
\ben[\rm(1)]
\item
 the map $\Phi$ is unital 
 if and only if in addition to \eqref{eq:tausdp}
 \beq
 \label{eq:unitaltr}
\sum_{p=1}^{n} C_{p,p} = I_{m};
 \eeq
 \item
 the map $\Phi$ is a quantum channel
 if and only if in addition to \eqref{eq:tausdp}
 \beq 
  \label{eq:unitaltr2}
 (\tr (C_{p,q}))_{p,q} = I_n; 
 \eeq
 \item  
  the map $\Phi$ is a quantum operation
 if and only if, in addition  to \eqref{eq:tausdp},
 \beq 
  \label{eq:unitaltr3}
( \tr (C_{p,q}))_{p,q} \preceq I_n.
 \eeq
\een
  In each case the constraints on $C$ are LMIs, and
  the set of  solutions $C$
constitute  a bounded spectrahedron.
       \end{theorem}
  
  \begin{rem} 
In the unital case the obtained  spectrahedron is free.
Namely, for fixed $A^1\in\smatng$, 
     the sequence of solution sets 
     to \eqref{eq:tausdp} and \eqref{eq:unitaltr}
     parametrized over $m$ is a free spectrahedron.
     See Proposition \ref{prop:liftPolar} for details.
  In the  two quantum cases, for each $m$, the solutions $\cD(m)$ at level $m$ form a spectrahedron, but the sequence $\cD=(\cD(m))_m$ is
  in general not a free spectrahedron since it fails to respect direct sums.
  \end{rem}
 
\def\bep{\begin{proof}}
\def\eep{\end{proof}}

\bep 
This interpolation result is a consequence of Theorem \ref{thm:cpstuff}.
 Let $\cS$ denote the span of $\{A^1_\ell\}$  and
  $\phi$ the mapping from $\cS$ to $M_m$ defined by
  $\phi(A^1_\ell)=A^2_\ell$.  This mapping has a completely positive extension
 $\Phi:M_n\to M_m$ 
  if and only if it has a positive semidefinite Choi matrix. The conditions on 
  $C$ evidently are exactly those needed to say that $C$ is a positive semidefinite Choi
  matrix for $\phi$.

The additional conditions  in
\eqref{eq:unitaltr} and   \eqref{eq:unitaltr2}
  (i.e., $\phi(I_n)=I_m$ and  trace preservation)
are clearly linear, so produce a spectrahedron in $\mbS_{mn}$.  
Both spectrahedra are bounded. Indeed, in each case $C_{p,p}\preceq I_m$, 
so $C\preceq I_{mn}$. Likewise, the additional condition 
in \eqref{eq:unitaltr3} is an LMI constraint, producing a bounded spectrahedron.
  \eep

We note that cp maps between 
subspaces of matrix algebras in the absence of positive definite elements were
treated in \cite[Section 8]{HKN+}; see also \cite{KS13,KTT13}.

\section{Free Spectrahedrops and Polar Duals} 
\label{sec:pdual4}

 This section starts by recalling the definition of a  free spectrahedrop as the 
 coordinate projection of a spectrahedron. It then continues
  with a review of free polar duals \cite{EW97}
 and their basic properties  before turning to 
 two main results, stated now without technical hypotheses.
 Firstly,  
 a free convex set is, in a canonical sense, generated
 by a finite set (equivalently a single point) if and only if 
 it is the polar dual of a free spectrahedron (Theorem \ref{thm:polarLMI}). 
  Secondly, 
 the polar dual of a free spectrahedrop is again a free spectrahedrop (Theorem \ref{thm:polarPolar}).

\subsection{Projections of Free Spectrahedra: Free Spectrahedrops}\label{sec:4}

 Let $L$ be a linear pencil in the variables $(x_1,\dots,x_g;y_1,\dots,y_h)$. Thus, 
  for some $d$ and $d\times d$ hermitian matrices $D,\Omega_1,\dots,\Omega_g,\Gamma_1,\dots,\Gamma_h$, 
\[ 
 L(x,y) = D + \sum_{j=1}^g \Omega_j x_j +\sum_{\ell=1}^h \Gamma_\ell y_\ell.
\] 
 The set
\[
 \proj_x \cD_L(1) = \{x\in\mathbb R^g: \exists\, y\in\mathbb R^h \mbox{ such that }
    L(x,y) \succeq 0\}
\]
 is known as a \df{spectrahedral shadow} or a \df{semidefinite programming (SDP) representable set} \cite{BPR13}
  and the representation afforded by $L$ is an \df{SDP representation}.  SDP representable sets are evidently
 convex and 
 lie in a middle ground between LMI representable sets and general convex sets. 
 They play an important role in convex optimization \cite{Ne06}. 
 In the case that $S\subset \mathbb R^g$ is closed  semialgebraic and satisfies  some mild
 additional hypothesis, it is proved in \cite{HN10} based upon the Lasserre--Parrilo construction 
 (\cite{Las09,Par06})
that the convex hull of $S$ is SDP representable.

  Given a linear pencil $L$, let $\proj_x \cD_L=(\proj_x \cD_L(n))_n$ denote the free set 
\[
 \proj_x \cD_L (n) =\{X\in\smatng : \exists\, Y\in\mathbb S_n^h
  \mbox{ such that } L(X,Y)\succeq 0\}.
\]
  We call a set of the form $\proj_x\cD_L$ a \df{free spectrahedrop}
  and  $\cD_L$ an {\bf LMI lift} 
  of $\proj_x \cD_L$.
  Thus a free spectrahedrop is  a coordinate projection of a free spectrahedron.
    Clearly, free spectrahedrops are matrix convex. In particular, they  are closed with respect to restrictions to reducing subspaces. 

\def\calK{\mathcal K}

\begin{lem} [\protect{\cite[\S4.1]{Lasse}}]
\label{lem:boundedmonicLift}
 If $\calK =\proj_x \cD_L$  
 is a free spectrahedrop  containing $0\in\R^g$ in the interior of $\calK(1)$,
 then there exists a monic 
 linear pencil $\mL(x,y)$ such that 
\[ 
\calK= \proj_x \cD_{\mL} = \{ X\in\smatg : \exists Y\in\smath: \, \mL(X,Y)\succeq0\}.
\] 
If, in addition, $\cD_L$ is bounded, then 
 we may further ensure $\cD_{\mL}$ is bounded.
\end{lem}

If the free spectrahedrop $\calK$ is closed and bounded, and contains $0$ in its interior, then 
there is a monic linear pencil $\mL$ such that $\cD_{\mL}$ is bounded and $\calK =\proj_x \cD_{\mL}$. See Theorem \ref{thm:polarPolar}. 

Let $p=1-x_1^2-x_2^4$.  It is well known that $\cD_p(1)=\{(x,y)\in\mathbb R^2: 1-x_1^2 - x_2^4 \ge 0\}$ is a spectrahedral shadow.  On the other hand, $\cD_p(2)$ is not convex (in the usual sense) and hence $\cD_p$ is not a spectrahedrop.  Further details can be found in Example \ref{ex:btv}.

\subsection{Basics of Polar Duals}
 \label{sec:free polar dual basics}

By precise analogy with the classical $\RR^g$ notion, the \df{free polar dual} 
$\calK^\circ =(\calK^\circ(n))_n$ of a free set $\calK\subset\smatg$ is 
\[ 
 \calK^\circ(n):=  \{ A \in \smatng : \
 \mL_A(X)=  I \otimes I - \sum_j^g A_j \otimes X_j \succeq 0
\text{ for all } X  \in \calK \}.
\]

Given $\eps>0$, consider the free $\epsilon$ ball centered at $0$,  
\[
\cN_\eps:= \{ X\in \smatg : \|X\|\leq\eps\} = \Big\{X : \eps^2 I \succeq \sum_j X_j^2 \Big\}.
\]
It is easy to see that its 
 polar dual 
 is bounded. In fact,
\[
\cN_{\frac1{g\eps}} \subseteq
\cN_\eps^\circ\subseteq
\cN_{\frac{\sqrt g}{\eps}}.
\]

  We say that  \df{$0$ is in the interior} of the subset $\Gamma \subset \smatg$ if 
  $\Gamma$ contains some free $\epsilon$ ball centered at $0$.

\begin{lem}\label{lem:interiorSame}
Suppose $\calK\subset\smatg$ is matrix convex. The following are equivalent.
\ben[\rm(i)]
\item
$0\in\R^g$ is in the interior of $\calK(1)$;
\item
$0\in\smatng$ is in the interior of $\calK(n)$ for some $n$;
\item
$0\in\smatng$ is in the interior of $\calK(n)$ for all $n$;
\item
$0$ is in the interior of $\calK$.
\een
\end{lem}

\begin{proof}
It is clear that (iv) $\Rightarrow$ (iii) $\Rightarrow$ (ii).
Assume (ii) holds. There is an $\eps>0$ with $\cN_{\eps}(n)\subseteq\calK(n)$. 
Since $\calK$ is closed with respect to restriction to reducing subspaces, and
\[
\cN_{\eps}(1)\oplus \cdots \oplus \cN_{\eps}(1) \subseteq\cN_{\eps}(n),
\]
we see $\cN_{\eps}(1)\subseteq\calK(1)$, i.e., (i) holds.

Now suppose (i) holds, i.e., $\cN_{\eps}(1)\subseteq \calK(1)$ for some $\eps>0$.
We claim that $\cN_{\eps/g^2}\subseteq\calK$. 
Let $X\in\cN_{\eps/g^2}$ be arbitrary.
It is clear that
\[
\left[- \frac{\eps}g,\frac{\eps}g\right]^g\subseteq\calK(1),
\]
hence  $\left[ -\eps/g,\eps/g\right]^g\otimes I_n\subseteq\calK(n)$.
Since each $X_j$ has norm $\leq\eps/g^2$, matrix convexity of $\calK$ implies that
\[
(0,\ldots,0,g X_j,0,\ldots,0)\in\calK
\]
and thus
\[
X=
\frac1g 
\big( (gX_1,0,\ldots,0) + \cdots + (0,\ldots,0,gX_g)\big)
\in\calK.\qedhere
\]
\end{proof}

 For the readers' convenience, the following proposition lists some properties of $\calK^\circ.$
 The bipolar result of item \eqref{it:polarpolar} is due to \cite{EW97}. 
 Given $\Gamma_\alpha$, a collection of matrix convex sets, it is readily verified that
  $\Gamma=(\Gamma(n))_n$ defined by $\Gamma(n) = \bigcap_{\alpha} \Gamma_\alpha(n)$ is again
  matrix convex. Likewise, if $\Gamma$ is matrix convex, then so is its closure
  $\overline{\Gamma} = (\overline{\Gamma(n)})_n$.
 Given a subset $\calK$ of $\smatg$, 
 let $\mco \calK$ denote the intersection of all matrix convex sets
  containing $\calK$. Thus, $\mco \calK$ is the smallest matrix convex set containing $\calK$.
  Likewise, $\cmco \calK= \overline{\mco \calK}$ is the smallest closed matrix convex set 
  containing $\calK$.   Details, and an alternate characterization of the matrix 
  convex hull of a free set $\calK$,  can be found in \cite{Lasse}.

\begin{prop}
 \label{prop:poDual}
Suppose $\calK\subset\smatg$.
\ben    [  \rm (1)]
\item
\label{it:mcon}
$\calK^\circ$ is a  closed matrix convex set containing $0;$
\item
 \label{it:bounded}
   if $0$ is in the interior of $\calK,$ then $\calK^\circ$ is bounded;
\item
\label{it:circ}
  $\calK(n) \subset \calK^{\circ\circ} (n)$  for all $n;$ that is, $\calK\subset \calK^{\circ\circ}$;
\item
 \label{it:bounded2}
  $\calK$ is bounded if and only if
$0$ is in the interior of $\calK^\circ$;
\item
\label{it:circconv}
 if there is an $m$ such that $0\in \calK(m)$, then $\calK^{\circ\circ}  = \cmco \calK;$ 
\item
 \label{it:polarpolar}
   if $\calK$ is a closed matrix convex set containing $0$, then $\calK=\calK^{\circ\circ};$ and
\item
 \label{it:Kcirc1vsK1circ}
   if $\calK$ is matrix convex, then  $\calK(1)^\circ= \calK^\circ(1)$. 
\een
\end{prop}

\begin{proof}
Matrix convexity in \eqref{it:mcon} is straightforward.

If $\calK$ has $0$ in its interior, then
there is a small free neighborhood $\cN_\eps$ of $0$ inside $\calK$.  Hence
$\calK^\circ\subseteq \cN_\eps^\circ=\cN_{1/{\eps}}$ is bounded.

Item \eqref{it:circ}
 is a tautology. Indeed,
if $X \in \calK(n)$, then 
we want to show $\mL_X(A)  \succeq 0 $ 
whenever  $\mL_A(Y) \succeq 0 $ for all $Y$ in $\calK$.
But this  follows simply from the fact that $\mL_X(A)$ and $\mL_A(X)$
 are unitarily equivalent.

If $\calK$ is bounded, then it is evident that $0$ is in the interior of $\calK^\circ$.
 If $0$ is in the interior of $\calK^\circ$, then, by item \eqref{it:bounded}, $\calK^{\circ\circ}$ is bounded.
 By item \eqref{it:circ}, $\calK\subset \calK^{\circ\circ}$  and thus $\calK$ is bounded.

To prove \eqref{it:circconv}, first note that $0\in \cmco\calK(m)$ and
  since $\cmco\calK(m)$ is matrix convex, $0\in\cmco\calK(1)$. Now 
suppose 
$W \not \in  \cmco \calK$. The Effros-Winkler 
matricial Hahn-Banach Theorem \ref{prop:sharp}
produces a monic linear pencil $\mL_A$ (with the size of $A$ no larger than the size of $W$)
separating $W$ from $\cmco \calK$; that is,  $\mL_A(W) \not \succeq 0$
and $\mL_A(X)  \succeq 0$ for $X \in \mco \calK.$
 Hence  $A \in \calK^\circ$. 
 Using  the unitary equivalence of $\mL_W(A)$ and $\mL_A(W)$ it follows that $\mL_W(A)\not \succeq 0,$ 
 and thus $W\notin \calK^{\circ\circ}.$  Thus, $\calK^{\circ\circ}\subset \cmco \calK$. 
 The reverse inclusion follows from item \eqref{it:circ}.

 Finally, suppose $\calK$ is matrix convex and $y\in \calK(1)^\circ$. Thus, $\sum y_j x_j = \langle y,x\rangle \le 1$ for
  all $x\in \calK(1)$.  Given $X\in \calK(m)$ and a unit vector $v\in\C^m$, since $v^*Xv\in \calK(1)$,
\[
  1\ge   \sum y_j v^*X_j v.
\]
 Hence,
\[
 v^*\big(I-\sum y_j X_j\big)v \ge 0
\]
 for all unit vectors $v$. So $y \in \calK^\circ(1)$. 
  The reverse inclusion is immediate. 
\end{proof}

\begin{cor}
If $\calK\subset\smatg$, then $\calK^{\circ\circ}=\cmco \big(\calK\cup\{0\}\big)$.
Here $0\in\mathbb R^g$. 
\end{cor}

\begin{proof}
 Note that $\calK^\circ = (\calK \cup\{0\})^\circ$ and hence,
\[
 \calK^{\circ\circ} = (\calK \cup\{0\})^{\circ\circ}. 
\]
   By item \eqref{it:circconv} of Proposition \ref{prop:poDual},
\[
  \cmco \big(\calK\cup\{0\}\big) = (\calK\cup\{0\})^{\circ\circ}. \qedhere
\]
\end{proof}

\begin{lem}\label{lem:projDual}
Suppose $\calK\subseteq\mbS^{g+h}$, and consider its image $\proj \calK\subseteq\smatg$
under the projection $\proj:\mbS^{g+h}\to\smatg$. A tuple
$A\in\smatg$ is in $(\proj\calK)^\circ$ if and only if $(A,0)\in\calK^\circ$.
\end{lem}

\begin{proof}
Note that $A\in(\proj\calK)^\circ$ if and only if for all $X\in \proj\calK$ we have
$\mL_A(X)\succeq0$ if and only if $\mL_{(A,0)}(X,Y)\succeq0$ for all $X\in\proj\calK$ and all $Y\in\smath$
if and only if $\mL_{(A,0)}(X,Y)\succeq0$ for all $(X,Y)\in\calK$ if and only if
$(A,0)\in\calK^\circ$.
\end{proof}

The polar dual of the set $\{(x_1,x_2)\in\mathbb R^2 : 1-x_1^2 -x_2^4 \ge 0\}$  is computed and seen not to be a spectrahedron in Example \ref{ex:scalarbentTVpd}.

\subsection{Polar Duals of Free Spectrahedra}
The next theorem completely characterizes finitely generated matrix convex sets $\calK$ containing $0$ in their interior. 
Namely, such sets are exactly polar duals of bounded free spectrahedra.

\begin{thm}
\label{thm:polarLMI}
 Suppose $\calK$ is a closed matrix convex set with $0$ in its interior. 
 If there is an   $\Om\in \calK$  such that for
 each  $X \in \calK$ there is a $\mu\in\N$ and an isometry $V$ such that 
\beq\label{eq:univRep}
  X_j = V^* (I_\mu \otimes \Om_j) V,
\eeq
 then
\beq\label{eq:polarLMI}
 \calK^\circ= \cD_{\mL_\Om},
\eeq
 where $\mL_\Omega$ is the monic linear pencil $\mL_\Omega(x) = I -     \sum\Omega_j x_j$. 

 Conversely, if there is an $\Omega$ such that  \eqref{eq:polarLMI} holds, 
 then $\Omega \in \calK$ and, for each $X\in \calK,$ there is an isometry $V$ such that \eqref{eq:univRep} holds.
\end{thm}

A variant of Theorem \ref{thm:polarLMI} in which the condition that $0$
 is in the interior of $\calK$ is replaced by the weaker hypothesis
  that $0$ is merely in $\calK$ and of course with a slightly 
  weaker conclusion, is stated as a separate result, Proposition
 \ref{cor:polarLMI} below.

\begin{lemma}
 \label{lem:dominate-polar}
  Suppose $\Omega \in \smatdg$ and consider the monic linear pencil $\mL_\Omega = I-\sum \Omega_j x_j$.
  \ben[\rm(1)]
\item
  Let $\Omega^\prime =\Omega\oplus 0$ where $0\in\smatdg.$ 
  A tuple $X\in \smatg$ is in $\cD_{\mL_\Omega}^\circ$ if and only if there
  is an isometry $V$ such that
\[
  X_j = V^* (I\otimes \Om_j^\prime)V.
\]
  \item
   If $\cD_{\mL_\Omega}$ is bounded, then $X\in\smatg$ is in $\cD_{\mL_\Omega}^\circ$
  if and only if there is an isometry $V$ such that equation \eqref{eq:univRep} holds.   
\een
\end{lemma}

\begin{rem}\rm
 \label{rem:contractvaddzero}
 As an alternate of (2), $X\in\cD_{\mL_\Omega}^\circ$ if and only if there exists a contraction $V$ such 
  that equation \eqref{eq:univRep} holds. 
\end{rem}

\begin{proof}
  Note that $X\in \cD_{\mL_\Omega}^\circ$ if and only if $\cD_{\mL_\Omega}\subset \cD_{\mL_X}$. 
  Thus if $\cD_{\mL_\Omega}$ is bounded, then the result follows directly from the last part of Theorem \ref{thm:convexPos}. 
  On the other hand, if $X$ has the representation of equation \eqref{eq:univRep}, then 
  evidently $X\in \cD_{\mL_\Omega}^\circ$. 

  If $\cD_{\mL_\Omega}$ is not necessarily bounded and $X\in\cD_{\mL_\Omega}^\circ(m)$,  then, by  Theorem \ref{thm:convexPos}, 
\[
  X = \sum_{j=1}^\mu V_j^* \Omega V_j,
\]
 for some $\mu$ and operators $V_j:\C^m\to\C^n$ such that 
\[
  I-\sum V_j^* V_j \succeq 0.
\]
 There is a $\nu>\mu$ and $m\times n$ matrices
 $V_{\mu+1},\dots,V_\nu$ such that 
\[
 \sum_{j=1}^\nu  V_j^* V_j  = I.
\]
 For $1\le j\le\mu$, let 
\[
  W_j = \begin{pmatrix} V_j \\ 0 \end{pmatrix}
\]
  and similarly for $\mu <j\le \nu$, let $W_j = \begin{pmatrix} 0 &  V_j^* \end{pmatrix}^*$. 
 With this choice of $W$, note that $\sum W_j^* W_j  =I_m$ and 
\begin{equation}
 \label{eq:oplus0}
  \sum W_j^* \Omega^\prime_j W_j = \sum W_j^* (\Omega_j\oplus 0) W_j = \sum_{j=1}^\nu V_j^* \Omega_j V_j = X_j.
\end{equation}

 If $X$ has the representation as in equation \eqref{eq:oplus0} and 
 $\mL_\Omega(Y)\succeq 0$, then 
\[
 \mL_X(Y) = \sum_j (W_j\otimes I)^* \mL_{\Omega^\prime}(Y) (W_j\otimes I). 
\]
 On the other hand, 
\[
 \mL_{\Omega^\prime}(Y) = \mL_\Omega(Y) \oplus I \succeq 0. 
\]
 Hence $X\in\cD_{\mL_\Omega}^\circ$. 
\end{proof}

\begin{proof}[Proof of Theorem {\rm\ref{thm:polarLMI}}] 
Suppose first \eqref{eq:polarLMI} holds for some $\Om\in\smatng$. 
 Since $\cD_{\mL_\Omega}^\circ=\calK$ and evidently $\Omega\in\cD_{\mL_\Omega}^\circ$, it follows that 
 $\Omega \in \calK$. 
  Since $0$ is assumed to be in the interior of $\calK$, its polar dual $\calK^\circ=\cD_{\mL_\Omega}$ is bounded 
by Proposition \ref{prop:poDual}.
  Thus, if $X\in \calK=\cD_{\mL_\Omega}^\circ,$ then 
  by Lemma \ref{lem:dominate-polar}, $X$ has a representation as in equation \eqref{eq:univRep}.

Conversely, assume that
$\Omega\in\calK$ has the property that any $X\in \calK$ can be represented as in \eqref{eq:univRep}. 
Consider the matrix convex set 
\[
\Gamma=\big\{V^* (I_\mu \otimes \Omega)V: \mu\in\N,\, V^*V=I\big\}.
\] 
Since $\Omega \in \calK$, it follows that $\Gamma \subset \calK$. On the other hand,
 the hypothesis is that $\calK\subset \Gamma$. Hence $\calK=\Gamma$.   Now,
 for $\mL_X$ a  monic linear pencil, $\mL_X(\Omega)\succeq 0$ if and only if 
\[
 \mL_X\big(V^*(I_\mu \otimes \Omega) V\big) = (V\otimes I)^*\, \mL_X(I_\mu \otimes \Omega)\,(V\otimes I) \succeq 0
\]
 over all choices of $\mu$ and isometries $V$.  Thus, $X\in \calK^\circ$ if and only if $\mL_X(\Omega)\succeq 0$.
 On the other hand, $\mL_X(\Omega)$ is unitarily equivalent to $\mL_\Omega(X)$. Thus
 $X\in \calK^\circ$ if and only if $X\in\cD_{\mL_\Omega}$.
\end{proof}

\begin{prop}
\label{cor:polarLMI}
 Suppose $\calK$ is a closed matrix convex set containing $0$.
 If there is a   $\Om\in \calK$  such that for
 each  $X \in \calK$ there is a $\mu\in\N$ and an isometry $V$ such that 
\beq\label{eq:univRep'}
  X_j = V^* (I_\mu \otimes \Om_j') V,
\eeq
 then
\beq\label{eq:polarLMI'}
 \calK^\circ= \cD_{\mL_\Om},
\eeq
 where $\mL_\Omega$ is the monic linear pencil $\mL_\Omega(x) = I -     \sum\Omega_j x_j$.
 Here $\Omega'=\Omega\oplus0$ as in Lemma {\rm\ref{lem:dominate-polar}}.

 Conversely, if there is an $\Omega$ such that equation \eqref{eq:polarLMI'} holds, 
 then $\Omega \in \calK$ and for each $X\in \calK$ there is an isometry $V$ such that equation \eqref{eq:univRep'} holds.
\end{prop}

\begin{proof}
Suppose first \eqref{eq:polarLMI'} holds for some $\Om\in\smatng$. 
 Since $\cD_{\mL_\Omega}^\circ=\calK$ and evidently $\Omega\in\cD_{\mL_\Omega}^\circ$, it follows that 
 $\Omega \in \calK$.  By Lemma \ref{lem:dominate-polar}, if $X\in \calK=\cD_{\mL_\Omega}^\circ,$ then 
  $X$ has a representation as in equation \eqref{eq:univRep'}.
    
Conversely, assume that
$\Omega$ has the property that any $X\in \calK$ can be represented as in \eqref{eq:univRep'}. 
Consider the matrix convex set 
\[
\Gamma=\big\{V^* (I_\mu \otimes \Omega')V: \mu\in\N,\, V^*V=I\big\}.
\] 
Since $0,\Omega \in \calK$, it follows that 
$\Omega'=\Omega\oplus0\in\calK$ and thus
$\Gamma \subset \calK$. On the other hand,
 the hypothesis is that $\calK\subset \Gamma$. Hence $\calK=\Gamma$.   Now,
 for $\mL_X$ a  monic linear pencil, $\mL_X(\Omega)\succeq 0$ if and only if 
 $\mL_X(\Omega')\succeq0$ if and only if
\[
 \mL_X\big(V^*(I_\mu \otimes \Omega') V\big) = (V\otimes I)^*\, \mL_X(I_\mu \otimes \Omega')\,(V\otimes I) \succeq 0
\]
 over all choices of $\mu$ and isometries $V$.  Thus, $X\in \calK^\circ$ if and only if $\mL_X(\Omega)\succeq 0$.
 On the other hand, $\mL_X(\Omega)$ is unitarily equivalent to $\mL_\Omega(X)$. Thus
 $X\in \calK^\circ$ if and only if $X\in\cD_{\mL_\Omega}$.
\end{proof}

\begin{rem}\rm 
\mbox{}\par
\ben[\rm(1)]
\item
For perspective, in the classical (not free)
situation when $g=2$, it is known that
$K\subseteq\R^2$ has an LMI representation if and only if $K^\circ$
is a numerical range \cite{Hen10,HS12}.
It is well known that the polar dual of a spectrahedron is not necessarily a spectrahedron.
This is the case even in $\R^g$, cf.~\cite[Section 5]{BPR13} or Example \ref{ex:scalarbentTVpd}.
\item
In the commutative case the polar dual of a spectrahedron (more generally, of a spectrahedral shadow) is  a spectrahedral shadow, see \cite{GN11} or \cite[Chapter 5]{BPR13}.
\item
It turns out that the 
$\Om$ in Theorem \ref{thm:polarLMI} 
can be taken to be an extreme point of $\calK$ in a very strong free sense. 
We refer to \cite{Far00,Kls13,WW99} for more on matrix extreme points.\qedhere
\een
\end{rem}

\subsection{The Polar Dual of a Free Spectrahedrop is a Free Spectrahedrop}
This subsection contains a duality result for free spectrahedrops (Theorem \ref{thm:polarPolar}) and several of its corollaries. 

  It can happen that $\cD_\mL$ is not bounded, but the projection
   $\calK=\proj_x \cD_\mL$ is.  Corollary \ref{for:strata} says that a
   free spectrahedrop is closed and bounded if and only if it is 
   the projection of some bounded free spectrahedron.  
   For expositional purposes, it is convenient to introduce the following terminology. 
   A free spectrahedrop $\calK$ 
   is called \df{stratospherically bounded} if there is a 
  linear pencil $\mL$ such that $\calK=\proj_x \cD_\mL$, and $\cD_\mL$ 
  is bounded.

\begin{thm}
 \label{thm:polarPolar}
 Suppose $\calK$ is a closed matrix convex set containing $0.$  
\ben[\rm(1)]
\item 
 If $\calK$ is a free spectrahedrop
 and $0$ is in the interior of $\calK,$ then $\calK^\circ$ is a stratospherically  bounded free spectrahedrop.
\item
If $\calK^\circ$ is a free spectrahedrop containing $0$ in its interior, then
 $\calK$ is a stratospherically bounded free spectrahedrop. 
\een
 In particular, if $\calK$ is a bounded free spectrahedrop with $0$ in its interior, then both $\calK$ and $\calK^\circ$
 are stratospherically  bounded free spectrahedrops $($with $0$ in their interiors$)$.
\end{thm}

 Before presenting the proof of the theorem we state a few
 corollaries and Proposition  \ref{prop:liftPolar}  needed in the proof.
 
\begin{cor}
 \label{cor:liftPolar}
 Given  $\Omega \in \smatdg,$ let $\mL_{\Omega}$ denote the corresponding monic
 linear pencil. The free set $\cD_{\mL_\Omega}^\circ$ is a stratospherically bounded free spectrahedrop. 
\end{cor}

\begin{proof}
  The set $\cD_{\mL_{\Omega}}$ is (trivially) a free spectrahedrop with $0$ in its interior. Thus, by Theorem \ref{thm:polarPolar},
  $\cD_{\mL_{\Omega}}^\circ$ is a stratospherically bounded free spectrahedrop.
\end{proof}

\begin{cor}\label{for:strata}
A free spectrahedrop  $\calK\subset\smatg$ is  closed 
and bounded  if and only if it is stratospherically bounded.
\end{cor}

\begin{proof}
Implication $(\Leftarrow)$ is obvious. 
$(\Rightarrow)$ 
Let us first reduce to the case where $\calK(1)$ has nonempty interior.
If $\calK(1)$ has empty interior, then it is contained in a proper affine
hyperplane $\{\ell=0\}$ of $\R^g$. 
Here $\ell$ is an affine linear functional. 
In this case we can solve for one
of the variables thereby reducing the codimension of $\calK(1)$. 
(Note that $\ell=0$ on $\calK(1)$ implies $\ell=0$ on $\calK$, cf.~\cite[Lemma 3.3]{Lasse}.)

Now let $\hat x\in\R^g$ be an interior point of $\calK(1)$. Consider the translation
\beq\label{eq:transl}
\tilde\calK = \calK -\hat x = \bigcup_{n\in\N} \big\{ X- \hat x I_n : X\in\calK(n)\big\}.
\eeq
Clearly, $\tilde\calK$ is a bounded free spectrahedrop with $0$ in its interior. Hence
by Theorem \ref{thm:polarPolar}, it is stratospherically bounded. Translating back, we see  $\calK$ is a stratospherically bounded free spectrahedrop.
\end{proof}

Each stratospherically bounded free spectrahedrop is closed, since it is
the projection of a (levelwise) compact spectrahedron.
Hence a bounded free spectrahedrop $\calK$ will not be 
stratospherically bounded if it is not closed. For a concrete example, consider
the  linear pencil
\[
L(x,y)=\begin{pmatrix} 2-x & 1 \\ 1 & 2-y \end{pmatrix} \oplus \begin{pmatrix} 2+x \end{pmatrix},
\]
and let $\calK=\proj_x\cD_{L}$. Thus
\[
\calK= \big\{ X\in\mbS:  -2\preceq X\prec2\big\}
\]
is bounded but not closed.

\begin{prop}
 \label{prop:liftPolar}
  Given $\Omega \in \mathbb S_d^g$ and $\Gamma\in\mathbb S_d^h$, the sequence $\calK=(\calK(n))_n,$
\[
  \calK(n) =\big\{A \in \smatng:   
 A = V^* (I_\mu \otimes \Omega) V, \ \  0=V^*(I_\mu \otimes \Gamma)V \textrm{ for some isometry } V
 \textrm{ and }\mu\leq nd\big\},
\]
 is a stratospherically bounded free spectrahedrop

 Let $\mL_{(\Omega,\Gamma)}$ denote the monic linear pencil corresponding to $(\Omega,\Gamma)$. The free set
\[
  \cC = \big\{A : (A,0)\in\cD_{\mL_{(\Omega,\Gamma)}}^\circ\big\}
\]
 is a stratospherically bounded free spectrahedrop. 
\end{prop}

\begin{proof}
Let $\cS$ denote the span of $\{I,  \Omega_1,\ldots,\Omega_g, \ \Gamma_1,\dots,\Gamma_h\}$.
  Thus $\cS$ is an operator system in $M_d$ (the fact that $I\in \cS$ implies $\cS$ contains 
  a positive definite element).  Let  \[\phi:\cS\to M_n\] denote the linear mapping determined by
 \[I  \mapsto I,\quad \Omega_j  \mapsto A_j,\quad\text{ and }\quad\Gamma_\ell  \mapsto 0.\]
  Observe that, by Theorem \ref{thm:cpstuff},  $A\in\calK(n)$ if and only if 
  $\phi$ has a completely positive extension $\Phi:M_d\to M_n$.  Theorem \ref{thm:interp}
  expresses existence of such a $\Phi$ as a (unital) cp interpolation problem in terms of a
   free spectrahedron.  For the reader's convenience we write out this critical LMI explicitly.
  Let  $\omega_{pq}^j$ denote the $(p,q)$-entry of $\Omega_j$
 and $\gamma_{pq}^\ell$ the $(p,q)$ entry of $\Gamma_\ell$. For a complex matrix (or scalar) $Q$ we use $\hat Q$ to denote its real part and $i\check Q$ for its imaginary part. Thus $\hat Q = \frac 12 (Q+Q^*)$ and $\check Q = -\frac{i}{2}(Q-Q^*)$.

Now
$A$ is in $\calK(n)$  if and only if  there exists $n\times n$ matrices $C_{p,q}$ satisfying
\begin{enumerate}[(i)]
 \item $ \sum_{p,q=1}^d E_{p,q} \otimes C_{p,q}  \succeq0$;
 \item $\sum_{p=1}^{d} C_{p,p} = I_{n};$ 
 \item $\sum_{p,q=1}^{d}  \omega_{pq}^\ell C_{p,q}=  A_\ell$ \ for  $\ell=1,\ldots,g;$ and
 \item $\sum_{p,q=1}^d \gamma_{pq}^\ell C_{p,q} = 0$\  for $\ell=1,\dots,h.$
\end{enumerate}
Since the $C_{p,q}$ for $p\neq q$ are not hermitian matrices, we rewrite the system (i) -- (iv) into one
with hermitian unknowns $\hat C_{p,q}$ and $\check C_{p,q}$.
Property (i) transforms into 
\[
\sum_{p,q} (\hat E_{p,q}\otimes \hat C_{p,q} - \check E_{p,q} \otimes \check C_{p,q}) + i (\hat E_{p,q}\otimes \check C_{p,q} + \check E_{p,q} \otimes \hat C_{p,q})\succeq0, 
\]
i.e.,
\beq\label{eq:complex1}
\begin{split}
\sum_{p,q} \hat E_{p,q}\otimes \hat C_{p,q} - \check E_{p,q} \otimes \check C_{p,q} &\succeq0 \\
\sum_{p,q} \hat E_{p,q}\otimes \check C_{p,q} + \check E_{p,q} \otimes \hat C_{p,q} &= 0.
\end{split}
\eeq
In item (ii) we simply replace $C_{p,p}$ with $\hat C_{p,p}$,
\beq\label{eq:complex2}
\sum_{p=1}^{d} \hat C_{p,p} = I_{n}.
\eeq
Properties (iii) and (iv) are handled similarly to (i).   Thus
\beq\label{eq:complex3}
\begin{split}
 \sum_{p,q} \hat \omega_{pq}^\ell \hat C_{p,q}
- \check \omega_{pq}^\ell \check C_{p,q}  & =  A_\ell \\
 \sum_{p,q} \hat \omega_{pq}^\ell \check C_{p,q}
+ \check \omega_{pq}^\ell \hat C_{p,q}  & =  0,
\end{split}
\eeq
and
\beq\label{eq:complex4}
\begin{split}
 \sum_{p,q} \hat \gamma_{pq}^\ell \hat C_{p,q}
- \check \gamma_{pq}^\ell \check C_{p,q}  & =  A_\ell \\
 \sum_{p,q} \hat \gamma_{pq}^\ell \check C_{p,q}
+ \check \gamma_{pq}^\ell \hat C_{p,q}  & =  0.
\end{split}
\eeq

  Thus, $\calK$ is the linear image of the explicitly 
  constructed free spectrahedron (in the variables
  $\hat C_{p,q} $ and $\check C_{p,q}$) 
given by \eqref{eq:complex1} -- \eqref{eq:complex4}. 
 Moreover,
 items (i) and (ii) together imply  $0\preceq C_{p,p}\preceq I$. 
 It now follows that
 $\|\hat C_{p,q}\|, \|\check C_{p,q}\|\le 1$ for all $p,q$. Thus, this free spectrahedron is bounded.  
 It is now routine to verify that $\calK$ is a projection
of a bounded free spectrahedron and is thus
 a stratospherically bounded free spectrahedrop.

 Let $\Omega^\prime = \Omega\oplus 0$ and $\Gamma^\prime = \Gamma\oplus 0$ where $0\in\smatdg$.  Note that
\[
  \cD_{\mL_{(\Omega,\Gamma)}} = \cD_{\mL_{(\Omega^\prime,\Gamma^\prime)}}.
\]
  By Lemma \ref{lem:dominate-polar},  $(A,0)\in \cD_{\mL_{(\Omega,\Gamma)}}^\circ$ if and only if 
\[
  A\in \big\{B: \exists \, \mu\in\N \mbox{ and an isometry } V \mbox{ such that } B= V^*(I_\mu\otimes \Omega^\prime) V, \  0 = V^*(I_\mu \otimes \Gamma^\prime)V\big\}.
\]
  By the first part of the proposition (applied to the tuple $(\Omega^\prime,\Gamma^\prime)$), it follows that
 $\cC$ is a stratospherically bounded free spectrahedrop.
\end{proof}

We are now ready to give the proof of 
Theorem \ref{thm:polarPolar}.

\begin{proof}[Proof of Theorem {\rm\ref{thm:polarPolar}}]
  Suppose $\calK$ is a free spectrahedrop with $0$ in its interior.  By Lemma \ref{lem:boundedmonicLift}, 
 there exists  $(\Omega,\Gamma),$  a pair of tuples of matrices, such that 
\[ 
   \calK=\big\{X: \exists Y \mbox{ such that } (X,Y) \in\cD_{\mL_{(\Omega,\Gamma)}}\big\} =\proj_x \cD_{\mL_{(\Omega,\Gamma)}},
\] 
 where $\mL_{(\Omega,\Gamma)}(x,y)$ is the monic linear pencil associated to $(\Omega,\Gamma)$.

Observe,  $A \in \calK^\circ$ if and only if for each $X\in \calK$, 
\[
  \mL_A(X) \succeq 0.
\]
 Thus, $A\in \calK^\circ$ if and only if 
\[
 \mL_{(A,0)} (X,Y) \succeq 0
\]
 for all $(X,Y)\in \cD_{\mL_{(\Omega,\Gamma)}}$
 if and only if 
 \[
 (A,0)\in  \cD_{\mL_{(\Omega,\Gamma)}}^\circ.
 \]
  Summarizing, $A\in \calK^\circ$ if and only if $(A,0)\in\cD_{\mL_{(\Omega,\Gamma)}}^\circ$. Thus, by the second part of
  Proposition \ref{prop:liftPolar}, $\calK^\circ$ is a
stratospherically  bounded free spectrahedrop.

 Because $\calK$ contains $0$ and is a closed matrix convex set, 
$\calK^{\circ\circ}=\calK$ by Proposition \ref{prop:poDual}.
 Thus, if  $\calK^\circ$ is a free spectrahedrop with $0$ in its interior, then, by what has already been proved,
 $\calK^{\circ\circ}=\calK$ is a stratospherically bounded free  spectrahedrop. 

 Finally, if $\calK$ is a bounded free spectrahedrop with $0$ in its interior, then $\calK^\circ$ contains $0$ in its interior and is a stratospherically bounded free spectrahedrop.
 Hence, $\calK=\calK^{\circ\circ}$ is also a stratospherically bounded free spectrahedrop.  
\end{proof}

Note that the polar dual of a free spectrahedron is a matrix convex set generated by
a singleton (Theorem \ref{thm:polarLMI}) and is a free spectrahedrop by the above
corollary.

\begin{cor}
 \label{cor:polardualofdrop}
 Let $\mL$ denote the monic linear pencil associated with $(\Omega,\Gamma)$. If $\calK=\proj_x \cD_{\mL}$ is bounded, then its 
 polar dual is the free set given by
\[
 \begin{split}
 \calK^\circ(n) &=  \big\{A\in\smatng: (A,0)\in \cD_{\mL}^\circ\big\} \\ 
& = \big\{A\in\smatng: \exists \mu\in\N \mbox{ and an isometry } V \mbox{ s.t. } 
   A= V^* (I_\mu \otimes \Omega)V, \ \ 0 = V^*(I_\mu \otimes \Gamma)V\big\}.
\end{split}
\]
 Whether or not $\calK$ is bounded,  
its polar
  dual is the free set
\[
 \begin{split}
 \calK^\circ(n) & =  \big\{A\in\smatng: (A,0)\in \cD_{\mL}^\circ\big\} \\ 
 & = \big\{A\in\smatng: \exists \mu\in\N \mbox{ and an isometry } V \mbox{ s.t. } 
   A= V^* (I_\mu \otimes \Omega^\prime)V, \ \ 0 = V^*(I_\mu \otimes \Gamma^\prime)V\big\},
\end{split}
\]
where $\Omega'=\Omega\oplus 0$ and $\Gamma'=\Gamma\oplus 0,$ as in Lemma 
{\rm\ref{lem:dominate-polar}}.
 \end{cor}

\begin{proof}
 From the proof of Theorem \ref{thm:polarPolar}, $\calK^\circ =\big\{A: (A,0)\in\cD_{\mL}^\circ\big\}$.
 Writing $\mL=\mL_\Delta$,  by Lemma \ref{lem:dominate-polar}
(whether or not $\cD_{\mL}$ is bounded), 
\[
 \cD_{\mL}^\circ=\big\{X: \exists \mu\in\N \mbox{ and an isometry } V \mbox{ such that } X= V^*(I_\mu\otimes \Delta^\prime)V\big\}.
\]

To obtain the stronger conclusion under the assumption that $\calK$ is bounded, 
an additional argument along the lines of \cite[\S 3.1]{HKM+} is needed; see also \cite[Theorem 2.12]{Za+}. 
Let $(A,0)\in\cD_\cL^\circ$.
We need to show that the unital linear map
\[
\begin{split}
\tau: {\rm span} \{I,\Omega_1,\ldots,\Omega_g,\Gamma_1,\ldots,\Gamma_h\} & \to {\rm span}\{I,A_1,\ldots,A_g\} \\
\Gamma_j\mapsto A_j, \quad \Gamma_k &\mapsto 0
\end{split}
\]
is completely positive. Assume 
\beq\label{eq:X0}
I\otimes X_0 + \sum_j \Omega_j \otimes X_j + \sum_k \Gamma_k \otimes Y_k \succeq0
\eeq
for some hermitian $X_0,\ldots,X_g,Y_1,\ldots,Y_h$. In particular, $X_0=X_0^*$.
We claim that $X_0\succeq0$. Suppose $X_0\not\succeq0$. By compressing we  may reduce to
$X_0\prec0$. From \eqref{eq:X0} it now follows that
\[
I\otimes tX_0 + \sum_j \Omega_j \otimes tX_j + \sum_k \Gamma_k \otimes tY_k \succeq0
\]
for every $t>0$. Since $tX_0\prec0$, this implies
\[
I\otimes I + \sum_j \Omega_j \otimes tX_j + \sum_k \Gamma_k \otimes tY_k \succeq0,
\]
whence
\[
(tX_1,\ldots,tX_g)\in\calK
\]
for every $t>0$. If $(X_1,\ldots,X_g)\neq0$ this contradicts the boundedness of $\calK$.
Otherwise $(X_1,\ldots,X_g)=0$, and 
\[
\sum_k \Gamma_k\otimes Y_k \succ -I\otimes X_0 \succ 0.
\]
Hence for any tuple $(X_1,\ldots,X_g)$ of hermitian matrices of the same size as the $Y_k$,
\[
I\otimes I + \sum_j \Omega_j \otimes X_j + \sum_k \Gamma_k \otimes tY_k \succeq0
\]
for some $t>0$. This again contradicts the boundedness of $\calK$. Thus $X_0\succeq0$.

By adding a small multiple of the identity to $X_0$ there is no harm in assuming $X_0\succ0$.
Hence multiplying \eqref{eq:X0} by $X_0^{-\frac12}$ from the left and right 
yields the tuple $X_0^{-\frac12}(X_1,\ldots,X_g$, $Y_1,\ldots,Y_h)X_0^{-\frac12}\in\cD_{\cL}$.
Since $(A,0)\in\cD_\cL^\circ$,
\[
\cL_{(A,0)}\big(X_0^{-\frac12}(X_1,\ldots,X_g,Y_1,\ldots,Y_h)X_0^{-\frac12}\big) 
= I\otimes I + \sum_k A_k\otimes X_0^{-\frac12} X_k X_0^{-\frac12} \succeq0.
\]
Multiplying with $X_0^{\frac12}$ on the left and right gives
\[
I\otimes X_0 + \sum_k A_k\otimes X_k \succeq0,
\]
as required.
\end{proof}

To each subset $\Gamma\subseteq\smatg$ we associate its interior
$\inter \Gamma=(\inter\Gamma(n))_{n\in\N}$, where $\inter\Gamma(n)$ denotes the interior
of $\Gamma(n)$ in the Euclidean space $\smatng$. We say $\Gamma$ has nonempty interior if there is $n$ with $\inter\Gamma(n)\neq\emptyset$.

\begin{cor}\label{cor:closeSpDrop}
 If  $\calK\subseteq\smatg$ is a bounded free spectrahedrop,  
 then $\overline\calK$ is a free spectrahedrop.
\end{cor}

\begin{proof}
As in the proof of Corollary \ref{for:strata}, we may assume 
the interior of $\calK$ is nonempty.
This implies there is a $\hat x\in\R^g$ in the interior of $\calK(1)$. Consider  the translation $\tilde\calK=\calK-\hat x$ as in \eqref{eq:transl}. This is a free spectrahedrop containing $0$ in its interior.
Hence its closure $\overline{\tilde\calK} = \tilde\calK^{\circ\circ}$ is a free spectrahedrop by Theorem \ref{thm:polarPolar}. Thus so is $\overline\calK = \overline{\tilde\calK} +  \hat x$.
\end{proof}

\begin{cor}
 \label{cor:dual of drop is drop}
 If  $\calK\subseteq\smatg$ is a free spectrahedrop with nonempty interior,
 then $\calK^\circ$ is a free spectrahedrop.
\end{cor}

\begin{proof}
 We are assuming  $\calK= \proj \cD_{L_A}.$ Applying 
 Lemma  \ref{lem:projDual} gives
 \[\calK^\circ =  \{B\in\smatg :   (B,0) \in \cD_{L_A}^\circ \}.\]
So if we prove $\cD_{L_A}^\circ$ is a free spectrahedrop, then
\[
\calK^\circ = \proj \Big(  \cD_{L_A}^\circ \bigcap \big(  \smatg\otimes\{0\}^h \big) \Big)
\]
is the intersection of two free spectrahedrops,
so is a free spectrahedrop. Thus without loss of generality we may take $\calK= \cD_{L_A}$ and proceed.
We will demonstrate the corollary holds in this case as a consequence of the Convex Positivstellensatz,
Theorem \ref{thm:convexPos}.

 Suppose $\hat x\in\R^g$ is in
the interior of $\cD_{L_A}$. Without loss of generality we may assume 
\[
L_0= L_A(\hat x)\succ0,
\]
cf.~\cite{HV07}. Define the monic linear pencil
\[
\begin{split}
\mL(y) & = L_0^{-\frac12} L(y+\hat x) L_0^{-\frac12}  \\
& = I +\sum_{j =1}^g  L_0^{-\frac12} A_j  L_0^{-\frac12}  y_j.
\end{split}
\]

By definition, a tuple $\Omega\in\smatg$ is in $\cD_{L_A}^\circ$ if and only if
$\cD_{L_A}\subseteq\cD_{\mL_\Omega}$. Equivalently, with
\[
L(y)= \big( I + \sum_{j=1}^g \Omega_j \hat x_j\big) + \sum_{j=1}^g \Omega_j y_j,
\]
$\cD_{\mL} \subseteq \cD_{L}$. By Theorem \ref{thm:convexPos}, there is a 
$S\succeq0$ and matrices $V_k$ with
\[
L(y) = S + \sum_k V_k^* \mL(y) V_k.
\]
That is,
\[
\big( I + \sum_{j=1}^g \Omega_j \hat x_j\big) \succeq \sum_k V_k^* V_k,
\quad \text{and} \quad
\Omega_j = \sum_k V_k^*  L_0^{-\frac12} A_j  L_0^{-\frac12} V_k , \; j=1,\ldots,g.
\]
Equivalently, there is a completely positive mapping $\Phi$ satisfying
\begin{align}
\Phi\big(L_0^{-\frac12} A_j  L_0^{-\frac12}\big)& =\Omega_j, \; j=1,\ldots, g\label{eq:choi1}  \\
\Phi(I) & \preceq I + \sum_j \Omega_j \hat x_j.\label{eq:choi2}
\end{align}
As in Theorem \ref{thm:interp} we now employ the Choi matrix $C$. Conditions \eqref{eq:choi1}
translate into linear constraints on the block entries $C_{ij}$ of $C$. Similarly, \eqref{eq:choi2} transforms
into an LMI constraint on the entries of $C$. Thus $C$ provides a free spectrahedral lift
of $\cD_{L_A}^\circ$.
\end{proof}

\subsection{The Free Convex Hull of a Union}
 
 In this subsection we prove that the convex hull of a union of
 free spectrahedrops is again a free spectrahedrop.
 
 \begin{prop}
 \label{prop:hull of union}
 If $\cS_1,\ldots, \cS_t\subseteq\smatg$ are stratospherically bounded 
 free spectrahedrops and each contains $0$ in its interior,
 then 
$\cmco(\cS_1\cup\cdots\cup \cS_t)$ is a stratospherically bounded free spectrahedrop with $0$ in its interior.
 \end{prop}

\begin{proof}
Let $\calK=\cS_1\cup\cdots\cup \cS_t$. Then 
\[
\calK^\circ = \cS_1^\circ \cap\cdots\cap \cS_t^\circ.
\]
Since each $\cS_j$ is a stratospherically bounded free spectrahedrop with $0$ in interior,
the same holds true for 
$\cS_j^\circ$
by Theorem \ref{thm:polarPolar}.
It is clear that these properties are preserved under a finite intersection,
so $\calK^\circ$ is again a stratospherically bounded free spectrahedrop with $0$ in its interior. 
Hence, by Proposition \ref{prop:poDual}), 
\[
(\calK^{\circ})^{\circ} = \cmco \calK 
\]
is a stratospherically bounded free spectrahedrop with $0$ in its interior
by Theorem \ref{thm:polarPolar}.
\end{proof}

\section{Positivstellensatz for Free Spectrahedrops}\label{sec:PosSS}
This section focuses on  polynomials positive on a free spectrahedrop, 
extending our Positivstellensatz for  free polynomials positive on free spectrahedra,
Theorem  \ref{thm:convexPos}, to a
Convex Positivstellensatz for free spectrahedrops, Theorem \ref{thm:posDrop}.

Let $\mL$ denote a monic linear pencil of size $d$,
\beq\label{eq:d-pencil}
  \mL (x,y)= I +\sum_{j=1}^g \Omega_j x_j +\sum_{k=1}^h \Gamma_k y_k,
  \eeq
and let $\calK=\proj_x\cD_{\mL}$.  If $\mu$ is a positive integer and $q_\ell\in \C^{d\times\mu}\ax$ (and so are polynomials in the $x$ variables only),
and if  $\sum_{\ell} q_\ell(x)^* \Gamma_k q_\ell(x) =0$ for each $k$, then
\[
 \sum_{\ell} q_\ell^*(x) \mL(x,y) q_\ell(x) 
\]
 is a polynomial in the $x$ variables and is thus in $\C^{\mu\times \mu}\ax$. 
 For positive integers $\mu$ and $r$ we define the \df{truncated quadratic module}
 in $\C^{\mu\times \mu}\ax$  associated to 
$\mL$ and $\calK$ by 
\[ 
M_{x}^\mu(\mL)_r =
\Big\{ \sum_\ell q_\ell^* \mL q_\ell + \sigma\; : \; q_\ell \in \C^{d\times\mu} \ax_r, \,
\sigma \in\Sigma^\mu_{r}\langle x \rangle, \,\sum_\ell q_\ell^* \Gamma_k q_\ell =0 \text{ for all } k\Big\}.
\] 
 Here $\Sigma^\mu_{r}= \Sigma^\mu_{r}\langle x \rangle$ denotes the set
of all sums
of hermitian squares $h^*h$ for $h\in\C^{\mu\times\mu}\ax_r$.
 It is easy to see 
$M_{x}^\mu(\mL)= \bigcup_{r\in\N} M_{x}^\mu(\mL)_r$ is a 
quadratic module in $\C^{\mu\times\mu}\ax$. 

The main result of this section is the following Positivstellensatz:

\begin{thm}\label{thm:posDrop}
A symmetric polynomial $p\in\C^{\mu\times\mu}\ax_{2r+1}$ is positive
 semidefinite on $\calK$ if and only if  $p\in M_{x}^\mu(\mL)_{r}$.
\end{thm}

\begin{rem}\rm Several remarks are in order. 
\mbox{}\par
\ben[\rm (1)]
\item
In case there are no $y$-variables in $\mL$, Theorem \ref{thm:posDrop} reduces
to the Convex Positivstellensatz of \cite{HKM12}.
\item
If $r=0$, i.e., $p$ is linear, then Theorem \ref{thm:posDrop} 
reduces to 
Corollary \ref{cor:polardualofdrop}.
\item
A Positivstellensatz for commutative polynomials \emph{strictly positive} on spectrahedrops was established by Gouveia and Netzer in \cite{GN11}. A major distinction is that
the degrees of the $q_i$ and $\sigma$ in the commutative theorem behave very badly.
\item 
Observe that $\calK$ is in general not closed. Thus Theorem \ref{thm:posDrop} yields a ``perfect'' Positivstellensatz for certain non-closed sets.
\qedhere
\een
\end{rem}

\subsection{Proof of Theorem {\rm \ref{thm:posDrop}}}

We begin with some auxiliary results.

\begin{prop}\label{prop:closed}
With $\mL$ a monic  linear pencil as in \eqref{eq:d-pencil}, 
$M_x^\mu(\mL)_r$ is a closed convex cone in
the set of all symmetric polynomials in
 $\C^{\mu\times\mu}\ax_{2r+1}$.
\end{prop}

The convex cone property is obvious.
 For the proof  that this cone is closed,
 it is convenient to introduce a norm compatible
 with $\mL$. 

\def\cB{\mathcal B}
\def\bbS{\mathbb S}
\def\bes{\[}
\def\ees{\]}

 Given $\eps>0$, let 
\bes
\cB_\eps^g(n): =  \big\{ X\in\bbS_n^g : \|X\|\le \eps \big\}.
\ees
 There is an
 $\eps>0$ such that for all $n\in\N$, if $(X,Y)\in\bbS_n^{g+h}$ and $\|(X,Y)\|\le \eps$, 
then  $\mL(X,Y)
 \succeq \frac12$.
 In particular, $\cB^{g+h}_{\eps} \subseteq \cD_{\mL}$. 
Using this $\eps$ we norm matrix polynomials in $g+h$ variables by
\beq\label{eq:normMe}
  \vertiii{p(x,y)}:= \max\big\{\|p(X,Y)\| : (X,Y)\in\cB_\eps^{g+h} \big\}.
\eeq
(Note that by the nonexistence of polynomial identities for matrices of all sizes, $\vertiii{p(x,y)}=0$ iff $p(x,y)=0$;
cf.~\cite[\S2.5, \S1.4]{Row80}.
Furthermore,
 on the right-hand side of \eqref{eq:normMe}
the maximum is attained because
the bounded free semialgebraic set $\cB_\eps^{g+h}$ is levelwise compact and matrix convex; see \cite[Section 2.3]{HM04a} for details).
  Note that if  $f\in\C^{d \times \mu}\ax_{\beta}$ and if $\vertiii{ f(x)^* \mL(x,y) f(x)} \le
  N^2$,  then $\vertiii{f^* f}\le 2N^2$.

\begin{proof}[Proof of Proposition {\rm\ref{prop:closed}}]  
  Suppose $(p_n)$ is a sequence from $M_x^\mu(\mL)_r$
  that converges to some symmetric 
  $p\in\C^{\mu\times \mu}\ax$ of degree at most $2r+1$.   
  By Caratheodory's convex hull theorem (see e.g.~\cite[Theorem I.2.3]{Bar02}), there is an $M$ 
 such that
  for each $n$ there exist matrix-valued polynomials 
  $r_{n,i}\in\C^{\mu\times \mu}\ax_{r}$ and 
  $t_{n,i}\in\C^{d \times \mu}\ax_{r}$ such that
\[
   p_n = \sum_{i=1}^{M} r_{n,i}^* r_{n,i}
         + \sum_{i=1}^M t_{n,i}^* \, \mL(x,y)\, t_{n,i}.
\]
  Since $\vertiii{p_n} \le N^2$, it follows that $\vertiii{r_{n,i}} \le
  N$ and likewise $\vertiii{t_{n,i}^* \mL(x,y) t_{n,i}}\le N^2$.  
  In view of the
  remarks preceding the proof, we obtain $\vertiii{t_{n,i}}
  \le \sqrt{2} N$ for all $i,n$. 
  Hence for each $i$, the sequences
  $(r_{n,i})$ and $(t_{n,i})$ are bounded in $n$.  They thus have
  convergent subsequences.   
Passing to one of these subsequential limits finishes the  proof. 
\end{proof}

Next is a variant of the Gelfand-Naimark-Segal (GNS) construction.
 
\def\cX{\mathcal X}
\def\mt{\nu}

\begin{prop}
 \label{prop:gns}
  If $\la:\C^{\mt\times \mt}\ax_{2k+2}\to \C$ is a linear functional  
  that  is nonnegative on $\Sigma^\mt_{k+1}$ and  positive  on
  $\Sigma^\mt_k\setminus\{0\}$, 
  then there exists a tuple $X=(X_1,\dots,X_g)$ of
  hermitian operators on a Hilbert space $\cX$ of
  dimension at most $\mt \sigma_\#(k)=\mt \dim \C\ax_k$ and a vector 
  $\ga\in \cX^{\oplus \mt}$ such that
\begin{equation}
 \label{eq:LorX}
  \la(f)= \langle f(X)\ga,\ga\rangle 
\end{equation}
  for all $f\in\C^{\mt\times \mt}\ax_{2k+1}$, where $\langle\textvisiblespace, \textvisiblespace\rangle$
  is the inner product on $\cX$.  
  Further, if
  $\la$ is nonnegative on $M_x^{\mt}(\mL)_{k}$, then $X$ is in the
closure $\overline\calK$ of the
   free spectrahedrop $\calK$ coming
  from $\mL$.

  Conversely, if $X=(X_1,\dots,X_g)$ is a tuple of
  symmetric operators on a Hilbert space $\cX$ of
  dimension $N$,  
  the vector $\ga\in\cX^{\oplus \mt},$ and $k$ is a positive
  integer, then the linear
  functional  $\la:\C^{\mt\times \mt}\ax_{2k+2}\to\C$ defined by 
\[
  \la(f)= \langle f(X)\ga,\ga\rangle 
\]
  is nonnegative on $\Sigma^\mt_{k+1}$. 
  Further, if $X\in \overline\calK$, then $\la$ is nonnegative also on $M_x^{\mt}(\mL)_{k}$. 
\end{prop}

\begin{proof}
The first part of the forward direction is standard, see e.g.~\cite[Proposition 2.5]{HKM12}.
In the course of the proof one constructs $X_j$ as the operators of multiplication by $x_j$ on a Hilbert space $\cX,$  that, as a set, is 
$\C\ax_k^{1\times \nu}$ (the set of row vectors of length $\nu$ whose entries are polynomials of degree at most $k$).
 The vector space $\cX^{\oplus \mt}$ in which $\gamma$ lies is $\C\ax_k^{\nu\times\nu}$ and
 $\gamma$ can be thought of as the identity matrix in 
 $\C\ax_k^{\nu\times \nu}$.  Indeed, the (column) vector $\gamma$ has $j$-th entry the row vector with $j$-th entry the empty set (which plays the role of
 multiplicative identity) and zeros elsewhere. 

In particular, for $p\in\cX=\C\ax_k^{1\times \nu}$, we have $p=p(X)\gamma$.
Let $\sigma$ denote the dimension of $\cX$ (which turns out to be $\nu$ times the dimension of $\C\ax_k$).
 
We next assume  that $\la$ is nonnegative on $M_x^{\mt}(\mL)_{k}$ and claim
that then $X\in \ol\calK$. 
Assume otherwise. Then, as $\ol\calK$ is closed matrix convex (and $\calK$ contains $0$ since $\mL$ is monic), the matricial Hahn-Banach 
Theorem \ref{prop:sharp} applies: there is a monic linear pencil $\mL_\La$ of size $\sigma$ 
such that ${\mL_\La}|_{\calK} \succeq0$ and $\mL_\La(X)\not\succeq0$.
In particular, $\cD_{\mL_\La} \supseteq \calK$,
whence
\[
\cD_{\mL_\La}^\circ \subseteq \calK^\circ.
\]
By Corollary \ref{cor:polardualofdrop}, 
\begin{equation}
 \label{eq:LaDual}
  \calK^\circ(n) = 
   \Big\{ A \in \bbS_n^g :
 \exists\mu\in\N\ \exists \text{isometry } V: \, \sum_{j=1}^\mu V_j^* \Gamma V_j =0, \,
 \sum_{j=1}^\mu V_j^* \Omega V_j = A \Big\}.
 \end{equation}
 Since 
 $\La\in  \calK^\circ $, there is an isometry $W$ with
 \[
 \sum_{j=1}^\eta W_j^* \Gamma W_j =0, \quad
 \sum_{j=1}^\eta W_j^* \Omega W_j = \La.
 \]
 Here, $W={\rm col}(W_1,\ldots, W_\eta)$ for some $\eta$, and $W_j\in\C^{d\times\sigma}$.

Since $\mL_\La(X)\not\succeq0$, there is $u\in\C^\sigma\otimes\cX$ with
\beq\label{eq:uNot}
u^*L_{\La}(X)u<0.
\eeq
Let
\[
u=\sum_i e_i \otimes v_i,
\]
where $e_i\in\C^\sigma$ are the standard basis vectors, and $v_i\in\cX$. 
By the construction of $X$ and $\gamma$, there is a polynomial $p_i\in\C\ax_k^{1\times \nu}$ with $v_i=p_i(X)\ga$.
Now \eqref{eq:uNot} can be written as follows:
\beq\label{eq:long}
\begin{split}
0 & >  u^* \mL_{\La}(X)u = \big( \sum_i e_i \otimes v_i\big)^* 
\mL_\La(X)
 \big( \sum_j e_j \otimes v_j\big) \\ 
& = 
\sum_{i,j,\ell} \big( e_i \otimes v_i\big)^* (W_\ell \otimes I)^* 
\mL(X,Y) \big(W_\ell \otimes I) \big( e_j \otimes v_j\big)\\
& = 
\sum_{i,j,\ell} \big( W_\ell e_i \otimes p_i(X)\ga \big)^* 
\mL(X,Y) \big(W_\ell e_j \otimes  p_j(X)\ga \big).
\end{split}
\eeq
Letting $\vec p_\ell(x) = \sum_j W_\ell e_j\otimes   p_j(x)\in\C^{d\times \nu}\ax_k$,
\eqref{eq:long} is further equivalent to
\beq\label{eq:concluding}
0> \sum_\ell\big( \vec p_\ell(X) \ga)^* \mL(X,Y)  \big( \vec p_\ell(X) \ga) = \la(q),
\eeq
where $q= \sum_\ell\vec p_\ell(x)^* \mL(x,y)  \vec p_\ell(x)$ is a matrix polynomial only in $x$ by \eqref{eq:LaDual}, and thus $q\in M_x^\mu(\mL)_k$. But now \eqref{eq:concluding} 
contradicts the nonnegativity of $\la$ on $M_x^\mu(\mL)_k$.

The  converse is obvious. 
\end{proof}

\begin{proof}[Proof of Theorem {\rm\ref{thm:posDrop}}] 
   Let ${\rm Sym}\,\C^{\mu\times\mu}\ax_{2r+1}$ denote the symmetric elements of $\C^{\mu\times\mu}\ax_{2r+1}$.  Arguing by contradiction, suppose  $p\in{\rm Sym}\, \C^{\mu\times\mu}\ax_{2r+1}$ and  $p|_{\calK}\succeq0,$ but $p\not\in M_x^\mu(\mL)_r$.
By the scalar Hahn-Banach theorem and Proposition \ref{prop:closed}, 
there is a strictly separating positive (real) linear functional $\la:{\rm Sym}\,\C^{\mu\times\mu}\ax_{2r+1}\to\R$  
nonnegative on $M_x^\mu(\mL)_r$. 
We first extend $\la$ to a (complex) linear functional on 
the whole  $\C^{\mu\times\mu}\ax_{2r+1}$ by sending $q+is\mapsto \la(q)+i \la(s)$ for symmetric $q,s$. 
 We then extend $\la$ to a linear functional
(still called $\la$)
on $\C^{\mu\times\mu}\ax_{2r+2}$ by mapping 
\[
E_{ij}\otimes u^*v \mapsto
\begin {cases} 0 & \text{ if } i\neq j \text{ or } u\neq v \\
C & \text{otherwise,}
\end{cases}
\]
where $
i,j=1,\ldots,\mu$,  and $u,v\in\ax$ are of length $r+1$.
For $C>0$ large enough, this $\la$ will be nonnegative on 
$\Sigma^\mu_{r+1}$.
Perturbing $\la$ if necessary, we may further assume $\la$ is strictly positive
on $\Sigma^\mu_{r}\setminus\{0\}$. 
Now applying Proposition \ref{prop:gns} yields a matrix tuple $X\in \ol\calK$ and a vector $\ga$
satisfying \eqref{eq:LorX} (with $k=r$). But then
\[
0>\la(p) = \langle p(X)\ga,\ga\rangle \geq0,
\]
a contradiction.
\end{proof}

\def\cpd{^{\circ c} }

\def\cp{{cp }}
\def\cpdot{{cp.}}

\section{Tracial Sets}\label{sec:q1}
While this papers original motivation arose from considerations of
free optimization as it appears in 
linear systems theory,
determining  the matrix convex hull of a free set has an analog in 
quantum information theory, see \cite{LP11}. In free optimization,
the relevant maps are completely positive and {\it unital} (ucp). In
quantum information theory, the relevant maps are completely positive 
and {\it trace preserving} (CPTP) or {\it trace non-increasing}. 
This section begins by recalling the two {\it quantum interpolation problems}
 from Subsection \ref{subsec:interp} 
 before 
  reformulating these problem in terms of {\it tracial hulls}.
Corresponding duality results are the topic of the next section.

Recall a {quantum channel}  is a \cp map $\Phi$ from $M_n$  to $M_k$ that  is trace preserving,  $$\tr( \Phi(X) ) = \tr (X).$$
The dual $\Phi'$ of $\Phi$ is the mapping from $M_k$ to $M_n$  defined by
$$ 
\tr( \Phi(X)  Y^*) = \tr(X  \Phi'(Y)^*).
$$

\begin{lemma}[\protect{\cite[Proposition 1.2]{LP11}}]
 $\Phi'$ is a quantum channel \cp if and only if $\Phi$ is unital \cpdot
\end{lemma}

Recall the \cp interpolation problem from Subsection \ref{sec:interpAlgor}.
It asks, given $A \in \smatng$ and given $B$ in $\mbS_m^g$:
does there exist a unital \cp map $\Phi:  M_n \to M_m $ such that  $B_j= \Phi(A_j)$ for $j=1, \ldots, g?$ 
 The set of solutions $B$ for 
 a given $A$ is the matrix convex hull of $A$. 
The versions of the interpolation problem  arising in quantum information theory \cite{Hardy, Klesse, NCSB}
 replace  unital with trace preserving or  trace non-increasing.
 Namely, 
does $B_j= \Phi(A_j)$ for $j=1, \ldots, g$ for some 
trace preserving (resp. trace non-increasing) \cp map $\Phi:  M_n \to M_m $?
 The set of all solutions $B$ for a given $A$ is the
 \df{tracial hull} of $A$. Thus, 
 \beq
   \label{eq:thull}
 \thull(A)= \{B:  \Phi(A)=B \ \ \mbox{for some trace preserving \cp  map } \Phi \}. 
 \eeq
 We define the  \df{contractive tracial hull} of a tuple $A$ by 
\[
 \cthull(A)= \{B:  \Phi(A)=B \ \ \mbox{for some \cp trace non-increasing map } \Phi \}.
\]

 The article \cite{LP11} determines when $B \in \thull(A)$ for $g=1$ (see Section 3.2). 
 For any $g \geq 0$  the paper
 \cite[Section 3]{AGprept} converts this problem to an LMI
 suitable for semidefinite programming;
 see Theorem \ref{thm:interp} here for a similar result.

 While the unital and trace preserving (or trace non-increasing) interpolation
  problems have very similar formulations,  tracial hulls possess far
  less structure than matrix convex hulls. Indeed,  as is easily seen,
  tracial hulls need not be convex (levelwise) and contractive tracial hulls
  need not be closed with respect to direct sums.  
Tracial hulls are studied in Subsection  \ref{sec:ts},
 and contractive tracial hulls  in Subsection \ref{sec:cts}.
 Section \ref {sec:tHBan}  contains
 ``tracial" notions of half-space and corresponding
 Hahn-Banach type separation theorems.

\subsection{Tracial Sets and Hulls}
\label{sec:ts}
 A  set $\cY\subseteq\mbS^g$ is \df{tracial} if
 $Y\in \cY(n)$ and if $C_\ell$ are $m\times n$ matrices such that
\[ 
 \sum C_\ell^\ast C_\ell = I_n,
\] 
 then $\sum C_j Y C_j^\ast \in \cY(m)$.  The \df{tracial hull} of a subset $\cS\subseteq\mathbb S^g$ is
  the smallest tracial set containing $\cS,$ 
  denoted
 $\thull(\cS)$.  Note that, in the case
  that $\cS$ is a singleton, this definition is consistent with the definition
  afforded by equation \eqref{eq:thull}.

 The following lemma is an easy consequence of a theorem of Choi, stated in \cite[Proposition 4.7]{Pau02}. 
It caps the number of terms needed in a convex combination to represent
a given matrix tuple $Z$ in the tracial hull of $T.$  Hence it is an analog of 
Caratheodory's convex hull theorem (see e.g.~\cite[Theorem I.2.3]{Bar02}).

\begin{lemma}
 \label{lem:boundsum}
   Suppose $T\in\smatng$ and $C_1,\dots,C_N$ are $m\times n$ matrices making $\sum C_\ell^\ast C_\ell =I_n.$
   If $Z=\sum_{\ell=1}^N C_\ell T C_\ell^\ast$, then there exists $m\times n$ matrices $V_1,\dots,V_{mn}$ such that
   $\sum V_\ell^* V_\ell = I_n$ and 
 \[
   Z =\sum_{\ell=1}^{mn} V_\ell T V_\ell^*. 
 \]
\end{lemma}

\begin{proof}
 The mapping $\Phi:M_n\to M_m$ defined by
\[
 \Phi(X) = \sum C_\ell X C_\ell^* 
\]
 is completely positive. Hence, by \cite[Proposition 4.7]{Pau02}, there exist (at most) $nm$  matrices $V_j:\mathbb C^m \to \mathbb C^n$ such that
\[
 \Phi(X) = \sum_{\ell=1}^{mn} V_\ell X V_\ell^*.
\]
 In particular,
\[
  Z=\Phi(T) = \sum V_\ell T V_\ell^*.
\]
 Further, for all $m\times m$ matrices $X$, 
\[
 \begin{split}
 \tr(X)  &  =  \tr\big(X\sum C_\ell^* C_\ell\big) 
   =  \tr\big(\sum C_\ell X C_\ell^*\big) 
   =   \tr\big(\Phi(X)\big)\\
 &  =  \tr\big(\sum V_\ell X V_\ell^*\big)
   =  \tr\big(X \sum V_\ell^* V_\ell\big).
 \end{split}
\]
 It follows that $\sum V_\ell^* V_\ell =I.$ 
\end{proof}

\begin{lemma}
 \label{lem:hullofS}
    For $\cS=\{T\}$ a singleton, 
\[
  \thull(\{T\}) = \big\{ \sum C_\ell T C_\ell^\ast: \sum C_\ell^\ast C_\ell =I\big\}.
\]
  Moreover, this set is closed $($levelwise$)$.

    The tracial hull of a subset $\cS\subseteq\mathbb S^g$ is 
\[
  \thull(\cS)  = \Big\{
\sum C_\ell T C_\ell^* :   \sum C_\ell^\ast C_\ell =I, \, T\in \cS\Big\}  = \bigcup_{T\in \cS} \thull(\{T\})
  .
\]   
 If $\cS$ is a finite set, then the  tracial hull of $\cS$ is closed. 
\end{lemma}

\begin{proof}
  The first statement follows from the observation that  $\{\sum C_\ell T C_\ell^\ast: \sum C_\ell^\ast C_\ell=I\}$ is tracial.
  
  To prove the moreover, suppose $T$ has size $n$ and suppose $Z^k$ is a sequence from $\mathcal Y(m)$.  By Lemma \ref{lem:boundsum}
  for each $k$ there exists $nm$ matrices $V_{k,\ell}$ of size $n\times m$ such that
\[
 Z^k = \sum_{\ell} V_{k,\ell} T V_{k,\ell}^*
\]
  and each $V_{k,\ell}$ is a contraction. Hence, by passing to a subsequence, we can assume, that for each fixed $\ell$, the
  sequence $(V_{k,\ell})_k$ converges to some $W_\ell$. Hence
  $Z^k$ converges to $Z=\sum_\ell W_\ell TW_\ell^*$. Also, 
  since $\sum_\ell V_{k,\ell}^* V_{k,\ell}=I$ for each $k$, 
we have  $\sum_\ell W_\ell^* W_\ell=I$, whence $Z\in\mathcal Y(m)$.

  To prove the second statement,  let $\cS\subset \smatg$ be given. 
  Evidently, $$\cS\subset \bigcup_{T\in \cS} \thull(\{T\})\subset \thull(\cS).$$
  Hence it suffices to show that $\bigcup_{T\in\cS} \thull(\{T\})$ is itself tracially convex. 
  To this end, suppose  $X\in \bigcup_{T\in \cS} \thull(\{T\})$ and
  $C_1,\dots,C_N$ with $\sum C_\ell^\ast C_\ell =I$ are given (and of the appropriate sizes).
  There is a $S\in \cS$ such that $X\in\thull(\{S\})$.  Hence, 
  by the first part of the lemma, $\sum C_\ell X C_\ell^\ast \in \thull(\{S\}) \subset\bigcup_{T\in\cS}\thull(\{T\})$
  and the desired conclusion follows.

  The final statement of the lemma follows by combining its first two assertions and using the fact that
  the closure of a finite union is the finite union of the closures. 
\end{proof}

\subsection{Contractively Tracial Sets and Hulls}
\label{sec:cts}
 A  set $\cY\subseteq \mbS^g$ is \df{contractively tracial} if
 $Y\in \cY(m)$ and if $C_\ell$ are $n\times m$ matrices such that
\begin{equation}
 \label{eq:sumCjleq}
 \sum C_\ell^\ast C_\ell \preceq I_m, \quad
\end{equation}
 then $\sum C_j Y C_j^\ast \in \cY(n)$.  
 Note that, in this case,  $\cY$ is closed under unitary conjugation and compression to subspaces, but not necessarily 
 direct sums.
 It is clear that intersections of contractively tracial sets are again contractively tracial.

In the case $\cS$ is a singleton,  the \df{contractive tracial hull} of a set $\cS,$  \index{cthull} 
 defined  as the smallest contractively tracial set containing $\cS$, 
 is consistent with our earlier definition  in terms of \cp maps.

\begin{lemma}\label{lem:cthullofS}
    The contractive tracial hull of a subset $\cS\subseteq\mathbb S^g$ is 
\[
  \cthull(\cS)  = \Big\{
\sum C_\ell T C_\ell^* :   \sum C_\ell^\ast C_\ell \preceq I, \, T\in \cS\Big\}  
= \bigcup_{T\in \cS} \cthull(\{T\}) 
  .
\]   
 If $\cS$ is a finite set, then the contractive tracial hull of $\cS$ is closed. 
\end{lemma}

\begin{proof}
Proof is the same as for Lemma \ref{lem:hullofS}, so is omitted.
\end{proof}

Tracial and contractively tracial sets are not necessarily convex, as Example 
\ref{ex:nocon} illustrates, and they are not necessarily free sets because they 
may not respect direct sums. Lemma \ref{lem:sums-convex} below explains the  relation between these two failings.
 Recall, a subset $\cY$ of $\smatg$ is levelwise convex if each $\cY(n)$ is convex (in the usual sense as a subset of $\smatng$).
 Say that $\cY$ is \df{closed with respect to convex direct sums} 
 if given $\ell$ and $Y^1,\dots,Y^\ell\in\cY$
 and given $\lambda_1,\dots,\lambda_\ell\ge0$ with $\sum \lambda_j \le 1$, 
\[
  \oplus_j \lambda_j Y^j \in \cY.
\]

\begin{lemma}
 \label{lem:sums-convex}
  If $\cY$ is contractively tracial, then 
   $\cY$ is levelwise convex if and only if $\cY$
   is  closed with respect to convex direct sums.
\end{lemma}

\begin{proof} 
  Suppose each $\cY(m)$ is convex. Given $Y^j\in \cY(m_j)$ for $1\le j\le \ell$, let 
  $m=\sum m_j$.  Consider, the block operator column 
  $W_j$ embedding  $\mathbb C^{m_j}$
  into $\mathbb C^m = \bigoplus_j \mathbb C^{m_j}$.  Note that $W_j^\ast W_j = I_{m_j}$ and thus contractively tracial implies
 $W_j Y^j W_j^\ast \in \cY(m)$. Hence, given $\lambda_j\ge 0$ with
  $\sum \lambda_j =1$, convexity of  
  $\cY(m)$  (in the ordinary sense), implies
\[
 \bigoplus_j \lambda_j Y^j = \sum \lambda_j W_j Y^j W_j^\ast \in \cY(m).
\]

 To prove the converse, suppose $Y^j \in\cY(n)$ and $m=\ell n$. In this case,
 $\sum W_j W_j^\ast = I_n$ and hence tracial implies,
\[
 \sum W_j^\ast \big(\bigoplus \lambda_j Y^j\big)  W_j 
 =\sum \lambda_j Y^j \in \cY(n).  \qedhere
\] 
\end{proof}

  \subsection{Classical  Duals of Free Convex Hulls and of  Tracial Hulls}
\label{sec:classical duals}

This subsection gives properties of the classical polar dual  
of  matrix convex hulls  and tracial hulls. 
 Real linear functionals $\lambda:\smatng\to \mathbb R$ 
 are in one-one correspondence with elements $B\in\smatng$ via the pairing,
\[
 \lambda(X) = \tr\big(\sum B_j X_j\big), \quad X=(X_1,\ldots,X_g).
\]
 Write $\lambda_B$ for this $\lambda$.  \index{$\lambda_B$}
 To avoid confusion with the free polar duals appearing earlier in this article, 
  let 
$\cU \cpd$ denote the   \df{conventional polar dual} of a subset $\cU \subset\smatng$. Thus, 
$$
 \cU \cpd = \{B \in \smatng :  \lambda_B(X) \leq 1 \text{ for  all } X \in \cU \}.
$$

\begin{lemma}  Suppose $A\in\smatng$.
\begin{enumerate}[\rm (i)]
\item $ \mco(A)\cpd = \big\{Y:  \thull(Y) \subset \{A\}\cpd \big\};$
\item $ \thull(A)\cpd   = \big\{Y :  \{A\}\cpd \supset \mco(Y)    \big\};$ and
\item $\thull(B) \subset  \thull(A)$ if and only if 
$\{A\}\cpd \supset \mco(Y)$ implies 
$\{B\}\cpd \supset \mco(Y).$
\end{enumerate}
\end{lemma}

\begin{proof}
The first formula:
\[
 \begin{split}   \mco(A)\cpd   & = \big\{Y :  1- \tr(\sum_j V_j^* A V_j \ Y) \geq 0, \ \sum_j V_j^*  V_j =I  \big\} \\
   & = \big\{Y :  1- \tr( A \ \sum_j   V_j  YV_j^*) \geq 0, \ \sum_j V_j^*  V_j =I  \big\} \\
  & = \{Y : 1- \tr( A    G)  \geq 0, \ G \in \thull(Y)   \} \\
    & =    \big\{Y :  \{ A\}\cpd  \supset  \thull(Y)     \big\}. 
 \end{split}
\]
The second formula:
\[
 \begin{split}
   \thull(A)\cpd & = \big\{Y : 1- \tr( \sum_j V_j^* A V_j \ Y) \geq 0, \  \sum_j V_j V_j^* =I \big\} \\
  & = \big\{Y : 1- \tr( A \ \sum_j  V_j \ Y V_j^*   ) \geq 0, \ \sum_j V_j V_j^* =I \big\} \\
  & = \big\{Y :  \{A\}\cpd \supset \mco(Y)    \big\}. 
 \end{split}
\]
The third formula:
 $\thull(B) \subset   \thull(A)$ if and only if 
$\thull(B)\cpd \supset    \thull(A)\cpd$ if and only if
\[
 \big\{Y : \{ B\}\cpd \supset \mco(Y)   \big\}  \supset 
  \big\{Y :  \{A\}\cpd \supset \mco(Y)  \big\},
\]
 if and only if $\{A\}\cpd  \supset \mco(Y)$ and $\{B\}\cpd \supset \mco(Y).$
\end{proof}

  \section{Tracial Spectrahedra and an Effros-Winkler Separation Theorem}
\label{sec:tHBan}

Classically, convex sets are delineated by half-spaces.
In this section a notion of half-space suitable in the tracial context
 --  we call these {\it tracial spectrahedra} -- are introduced. 
Subsection \ref{subsec:tHBan} contains a free Hahn-Banach   
separation theorem for tracial spectrahedra. 
The  section concludes with  
 applications of this Hahn-Banach theorem.
Subsection \ref{sec:tduals} suggests
 several notions of duality based on the tracial separation theorem
 from Subsection \ref{subsec:tHBan}. 
 Subsection \ref{sec:tcs} studies free (convex) cones.

\subsection{Tracial Spectrahedra}
 Polar duality considerations in the trace non-increasing context lead naturally to inequalities of the
 type,
\[
 I\otimes T - \sum_{j=1}^g B_j\otimes Y_j \succeq 0,
\]
 for tuples $B,Y\in \smatg$ and a positive semidefinite matrix $T$ with trace at most one. 
 Two notions, in a sense dual to one another,  of half-space are obtained by fixing either $B$ or $Y$. 
 
 Given $B\in\mbS_k^g,$  let 
\[
\begin{split}
 \fH_B & =\bigcup_{m\in\N} \big\{ Y\in\smatmg: \exists T\succeq 0, \ \tr(T)\le 1, \ \  I\otimes T- \sum B_j\otimes Y_j \succeq 0\big\} \\
 &= \bigcup_{m\in\N}  \big\{ Y\in\smatmg: \exists  T\succeq 0, \ \tr(T)= 1,  \ \  I\otimes T- \sum B_j\otimes Y_j \succeq 0\big\}.
 \end{split}
\]
 We call sets of the form $\fH_B$  \df{tracial spectrahedra}. 
  Tracial spectrahedra obtained by fixing $Y$, and parameterizing  over $B$, appear in Subsubsection \ref{sec:Idem}. 

\begin{proposition}
 \label{prop:easytrace}
   Let $B\in\mbS_k^g$ be given.
 \begin{enumerate}[\rm (a)]
   \item The set $\fH_B$ is contractively tracial;
   \item For each $m$, the set $\fH_B(m)$ is convex; and
  \item For each $m$, the set $\fH_B(m)$ is closed.
\end{enumerate}
      In summary, $\fH_B$ is levelwise compact and closed, and  is contractively tracial. 
\end{proposition}

\begin{rem}\rm
 Of course $\fH_B$ is not a free set since, in particular, it is not closed with respect to direct sums. \qedhere

\end{rem}

\begin{proof}
 Suppose $Y\in \fH_B(m)$ and $C_\ell$ satisfying equation \eqref{eq:sumCjleq} are given. There is an $m\times m$ positive semidefinite matrix $T$ with trace at most one such that
\[
 I\otimes T -\sum B_j \otimes Y_j \succeq 0.
\]
 It follows that
\[
 0\preceq I\otimes \sum_\ell C_\ell T C_\ell^\ast - \sum_j B_j \otimes \sum_\ell C_\ell Y_j C_\ell^\ast.
\]
 Note that $T^\prime  = \sum_\ell C_\ell T C_\ell^\ast \succeq 0$ and 
\[
 \tr(T^\prime) = \tr\big(T \sum C_\ell^\ast C_\ell\big)
  = \tr\big(T^{\frac12} C_\ell^\ast C_\ell T^{\frac12}) \le \tr(T) \le 1.
\]
Hence $\sum C_\ell Y C_\ell^\ast \in\cY(n)$ and item (a) of the proposition is proved.

 To prove item (b), suppose both $Y^1$ and $Y^2$ are in $\mathcal Y_B$. To each there is an
  associated positive semidefinite matrix of trace at most one, say $T_1$ and $T_2$. 
  If $0\le s_1,s_2 \le 1$ and $s_1+s_2=1$, then $T=\sum s_\ell T_\ell$ is positive semidefinite and has trace at most one.
  Moreover, with $Y=\sum s_j Y^j$, 
\[
 I\otimes T -  \sum_j B_j\otimes \big(\sum s_\ell Y^\ell_j\big) = \sum_\ell s_\ell \big( I\otimes T_\ell - \sum_j B_j\otimes Y^\ell_j\big) \succeq 0. \]

 To prove (c), suppose the sequence  $(Y^k)_k$ from $\fH_B(m)$ converges to $Y\in\mbS_m^g$.
 For each $k$ there is a positive semidefinite matrix $T_k$ of trace at
 most one such that 
 \[
 I\otimes T_k - \Lambda_B(Y^k)\succeq0.
 \]
Choose a convergent subsequence of the $T_k$ with limit $T$. This $T$ witnesses $Y\in \fH_B(m)$. 
\end{proof}

To proceed toward the separation theorem we start with some preliminaries.

\subsection{An Auxiliary Result}
 Given a positive integer $n$, let
 $\cTn$ denote the positive semidefinite
 $n\times n$ matrices 
 of trace one.  Each $T \in \cTn$ corresponds
 to a state on $M_n$  via the
 trace,
\begin{equation}\label{eq:trstate}
  M_n \ni A \mapsto \tr(AT).
\end{equation}
Conversely, to each state $\varphi$ on $M_n$ we can assign a matrix $T$
such that $\varphi$ is the map \eqref{eq:trstate}.
  Note that $\cTn$ is a convex, compact subset of $\mathbb S_n,$
  the symmetric $n\times n$ matrices.  

 The following
 lemma is a version of  \cite[Lemma 5.2]{EW97}.
 An affine (real) linear mapping $f:\mathbb S_n \to \mathbb R$ is a
 function of the form $f(x)=a_f +\lambda_f(x)$,
 where $\lambda_f$ is (real) linear and $a_f\in\mathbb R$.

\begin{lemma}
 \label{lem:cone}
  Suppose $\mathcal F$ is a convex set of affine linear
  mappings $f:\mathbb S_n \to \mathbb R$. If for each $f\in \mathcal F$
  there is a $\hh \in\cTn$ such that $f(\hh)\ge 0$,
  then there is a $\hhs\in \cTn$ such that
  $f(\hhs)\ge 0$ for every $f\in\mathcal F$.
\end{lemma}

\begin{proof}
   For $f\in\mathcal F$, let
 \[
   B_f =\{\hh\in \cTn: f(\hh)\ge 0\}\subseteq \cTn.
 \]
  By hypothesis each $B_f$ is non-empty and
  it suffices to prove that
 $$
   \bigcap_{f\in\mathcal F} B_f \neq \emptyset.
 $$
  Since each $B_f$ is compact, it suffices to
  prove that the collection $\{B_f: f\in\mathcal F\}$
  has the finite intersection property.  Accordingly,
  let $f_1,\dots,f_m\in\mathcal F$ be given. Arguing
  by contradiction, suppose
\[
  \bigcap_{j=1}^m B_{f_j} =\emptyset.
\]
Define $F:\mathbb S_n \to \mathbb R^m$  by
\[
  F(\hh)=(f_1(\hh),\dots,f_m(\hh)).
\]
 Then $F(\cTn)$ 
 is both convex and compact because $\cTn$
 is both convex and compact and each
  $f_j$, and hence $F$, is affine linear.
 Moreover, 
  $F(\cTn)$  does not intersect 
 \[
   \mathbb R^m_{\geq0}=\{x=(x_1,\dots,x_m)\in\R^m: x_j\ge 0 \mbox{ for each } j\}.
 \]
 Hence there is a linear functional $\lambda:\mathbb R^m \to \mathbb R$
 such that 
 \[\lambda\big(F(\cTn)\big)<0 
 \quad\text{ and }\quad
\lambda\big(\mathbb R_{\geq0}^m\big) \ge0.
\]
 There exists $\lambda_j \in \R$ such that
$
 \lambda(x) = \sum \lambda_j x_j.
$
  Since $\lambda\big(\mathbb R^m_{\geq0}\big) \ge 0$ it follows that each
  $\lambda_j\ge 0$ and since $\lambda\ne 0$, for at least one $k$,
  $\lambda_k>0$.  Without loss of generality, 
  it may be assumed that $\sum \lambda_j=1$. 
  Let
\[
  f=\sum \lambda_j f_j.
\]
 Since $\mathcal F$ is convex, 
 it follows that $f\in\mathcal F$. On the other hand,
 $f(T)=\lambda(F(T)).$  Hence 
 if $T\in \cTn,$ then  $f(T)<0$.
 Thus, for this $f$ there does not exist
 a $T\in \cTn$ such that $f(T)\ge 0$,
 a contradiction which completes the proof.
\end{proof}

\subsection{A Tracial Spectrahedron Separating Theorem}
\label{subsec:tHBan}
 The following lemma is proved by a variant of the Effros-Winkler construction 
 of separating LMIs (i.e., the matricial Hahn-Banach Theorem) in the theory of matrix convex sets.

\begin{lemma}
 \label{lem:separate}
   Fix positive integers $m,n$, and  suppose that $\cS$ is a nonempty 
 subset of $\smatmg$.  Let $\cU$ denote the subset of $\smatng$  
  consisting of all tuples of the form
\[
 \sum_{\ell=1}^\mu C_\ell Y^\ell C_\ell^*,
\]
 where each $C_\ell$ is $n\times m$, each $Y^\ell\in \cS$ and 
  $\sum C_\ell^\ast C_\ell \preceq I$.  If $B\in\smatng$ is in
  the conventional polar dual of $\cU$, then there exists 
  a positive semidefinite $m\times m$ matrix $T$ with trace at most one such that
\[
   I\otimes T -\sum B_j\otimes Y_j \succeq 0
\]
 for every $Y\in\cS$. 
\end{lemma}

\begin{proof}
 Recall the definition of $\la_B$ from Subsection \ref{sec:classical duals}.
  Given $C_\ell$ and $Y^\ell$ as in the statement of the lemma,  define
  $f_{C,Y}: \smatmg\to \mathbb R$ by 
\[
  f_{C,Y}(X) = \tr\big(\sum C_\ell X C_\ell^\ast\big) - \lambda_B\big(\sum C_\ell YC_\ell^\ast\big).
\]
Let $\mathcal F=\{f_{C,Y}: C,Y\}.$ Thus $\mathfrak F$ is a set of affine (real) linear mappings from $\smatmg$ to $\mathbb R$.  To show that $\mathcal F$ is convex, suppose, for $1\le s \le N$, 
 $C^s=(C^s_1,\dots,C^s_{\mu_s})$ is a tuple of $n\times m$ matrices, 
 for $1\le s \le N$ and $1\le j \le \mu_s$ the matrices  $Y^{s,j}$
  are in  $\cS$ and 
 and $\lambda_1,\dots,\lambda_N$ are positive numbers with $\sum \lambda_s =1$.
 In this case, 
\[
 \sum \lambda_s f_{C^s,Y^{s,\cdot}} = f_{C,Y}
\] 
 for
\[
  C = \Big(\frac1{\sqrt{\lambda_s}} C_\ell^s\Big)_{s,\ell}, \ \ \ Y= \Big( Y^{s,\ell} \Big)_{s,\ell}.
\]
Hence $\mathcal F$ is convex. 
 
 Given $n\times m$ matrices  $C_1,\dots,C_\mu$ 
 and $Y^1,\dots,Y^\mu\in\mathcal S$, let $D=\sum C_\ell^\ast C_\ell$. 
 Assuming  $D$ has norm one, there is a unit vector $\gamma$ such that
  $\|D\gamma\| =\|D\|=1$.  Choose $T=\gamma \gamma^\ast$. Thus $T\in\mathcal T_m$. Moreover,
\[
 \tr\big(\sum C_\ell T C_\ell^\ast\big) =\tr(TD) = \langle D\gamma,\gamma\rangle =1.
\]
 Thus, using the assumption that $B$ is in $\cU\cpd$, 
\[
 f_{C,Y}(T) = 1 -\lambda_B\big(\sum C_\ell Y^\ell C_\ell^\ast\big) \ge 0.
\]
 If $D$ is not of norm one, a simple scaling argument gives the same conclusion; that is,
\[
  f_{C,Y}(T) \geq 0. 
\]
 Thus, for each $f\in \mathcal F$ there exists a $T\in \mathcal T_m$ such that $f(T)\ge 0$. 
 By Lemma \ref{lem:cone}, it follows that there is a $\mathfrak{T}\in \cT_m$ such that
 $f_C(\mathfrak{T})\ge 0$ for all $C$ and $Y$; i.e.,
\begin{equation}
 \label{eq:EWineq}
  \tr\big(\sum C_\ell \mathfrak{T} C_\ell^\ast\big) - \lambda_B\big(\sum C_\ell Y^\ell C_\ell^\ast\big) \ge 0,
\end{equation}
 regardless of the norm of $\sum C_\ell^\ast C_\ell$.

 Now the aim is to show that 
\[
  \Delta:=I\otimes \mathfrak{T} - \sum_j B_j\otimes Y_j  \succeq 0
\]
 for every $Y\in\cS$. 
 Accordingly, let $Y\in \cS$ and  $\gamma = \sum e_s \otimes \gamma_s \in \mathbb R^n\otimes \mathbb R^m$ be given. 
 Compute,
\[
 \begin{split}
   \langle \Delta \gamma,\gamma \rangle  
    & = \sum_s \langle \fT \gamma_s,\gamma_s \rangle - \sum_j \sum_{s,t} (B_j)_{s,t} \langle Y_j \gamma_s,\gamma_t\rangle.
 \end{split}
\]
 Now let $\Gamma^\ast$ denote the matrix with $s$-th column $\gamma_s$.  Hence $\Gamma$ is $n\times m$ and
\[
 \begin{split}
   \lambda_B(\Gamma Y \Gamma^\ast) & =  \tr\big(\sum B_j (\Gamma Y_j \Gamma^\ast)\big) \\
 &  = \sum_j \sum_{s,t} (B_j)_{s,t} \langle Y_j\gamma_s,\gamma_t\rangle.
 \end{split}
\]
 Similarly,
\[
 \tr(\Gamma \fT \Gamma^\ast) =  \sum_s \langle  \fT \gamma_s,\gamma_s\rangle.
\]
 Thus, using the inequality  \eqref{eq:EWineq}, 
\[
  \langle \Delta \gamma,\gamma \rangle  = \tr(\Gamma \fT \Gamma^\ast) - \lambda_B(\Gamma Y\Gamma^\ast) \ge 0.
\]
 It is in this last step that the contractively tracial, not just tracial is needed, so that
  it is not necessary for $\Gamma^\ast \Gamma$ to be a multiple of the identity. 
\end{proof}

\begin{proposition}
 \label{prop:main?}
  If $\cY\subseteq\smatg$ is contractively tracial and if $B\in\smatng$ is in the conventional polar dual $\cY(n)\cpd$ of $\cY(n)$,
  then $\cY\subseteq\fH_B$.
\end{proposition}

\begin{proof}
   Suppose $\cY$ is contractively tracial and $Y\in \cY(m)$.  Letting $\cS=\{Y\}$ in Lemma \ref{lem:separate}, 
   it follows that there is a $T$ such that
\[
  I\otimes T- \sum B_j\otimes Y_j\succeq 0.
\]
 Thus, $Y\in \fH_B$ and the proof  is complete.
\end{proof}

 We are now ready to state the separation  result for closed
 levelwise convex tracial sets.

\begin{theorem}
 \label{thm:tracehull} \mbox{}
 \begin{enumerate}[\rm (i)]
   \item
   If $\cY\subset\smatg$ is contractively tracial, levelwise convex, and 
   if $Z \in \mathbb S_m^g$ is not in the closure of $\mathcal Y(m)$, then there exists a $B\in \mathbb S_m^g$ such that
   $\mathcal Y\subseteq \fH_B$, but $Z\notin\fH_B$. Hence, 
\[
 \overline{\cY}  =\bigcap\{\fH_B: \fH_B \supseteq \mathcal Y\} = \bigcap_{n\in\N}\bigcap_{B\in \cY(n)\cpd}  \fH_B.
\]
  \item
      The levelwise closed convex contractively tracial hull of a subset $\cY$ of $\smatg$ is 
\[
  \bigcap\{\fH_B: \fH_B \supseteq \mathcal Y\}. 
\]
\end{enumerate}
\end{theorem}

\begin{proof}  
 To prove item (i), 
 suppose $Z\in\smatmg$ but $Z\notin \overline{\cY(m)}$. 
 Since $\cY$ is levelwise convex, there is $\lambda_B$ such that $\lambda_B(Y)\le 1$
 for all $Y\in \cY(m)$, but $\lambda_B(Z)> 1$ by the usual
  Hahn-Banach separation theorem for closed convex sets. 
   Thus $B$ is in the conventional polar dual of $\cY(m)\cpd$. 
  From Proposition \ref{prop:main?}, $\cY\subset \fH_B$. 

 On the other hand, if  $T\in\cT_m$ and $\{e_1,\dots,e_m\}$ is an orthonormal basis for $\mathbb R^m,$
  then,  with $e=\sum e_s\otimes e_s\in\mathbb R^m\otimes \mathbb R^m$,
\[
 \big\langle (I\otimes T - \sum B_j\otimes Z_j)e,e\big\rangle =
   \tr(T) - \tr\big(\sum B_j Z_j\big) \ = \ 1 - \lambda_B(Z) < 0.
\]
 Hence $Z\notin \fH_B$ and the conclusion follows. 

To prove item (ii), first note, letting $\mathcal I$ denote the intersection
  of the $\fH_B$ that contain $\cY$, that $\cY\subset\mathcal I$. Since  the intersection of tracial spectrahedra is 
  levelwise closed and convex, and contractively tracial, the levelwise closed convex tracial
  hull $\mathcal H$ of $\mathcal Y$ is also contained in $\mathcal I$. On the other hand, 
  from (i), 
\[
 \mathcal H = \bigcap \{\fH_B: \fH_B \supset \mathcal H\} \supset \mathcal I \supset \mathcal H.\qedhere
\]
\end{proof}

\begin{rem} {The contractive tracial hull of a point.}
 Fix a $Y\in\smatng$ and let $\cY$ denote its contractive tracial hull,
\[
 \cY=\big\{\sum V_j Y V_j^* : \sum V_j^* V_j \preceq I\big\}.
\]
 Evidently each $\cY(m)$ (taking $V_j:\mathbb R^n\to\mathbb R^m$) is a convex set.
 From Lemma \ref{lem:cthullofS}, $\cY$ is closed.  
  Hence Theorem \ref{thm:tracehull} applies and gives a duality description of $\cY$. 
 Namely, $\tilde{Y}$ is in the contractive tracial hull $\cY$ if and only if for each $B$ for which  there exists
 a positive semidefinite $T$ of trace at most one such that
\[
 I\otimes T - \sum B_j\otimes Y_j \succeq 0,
\]
 there exists a positive semidefinite  $\tilde{T}$ of trace at most one  such that 
\[
 I\otimes \tilde{T} - \sum B_j \otimes \tilde{Y}_j \succeq 0. \qedhere
\] 
\end{rem}

\subsection{Tracial Polar Duals}
\label{sec:tduals}

 We now introduce two natural notions of polar duals based on the
  tracial spectrahedra. Rather than  exhaustively
  studying these duals, we list a few properties 
  to illustrate the possibilities.

\subsubsection{\Contra Tracial Dual}
  Suppose $\calK\subset \smatg$. Let $\hK$ denote its 
  \df{\contra tracial dual}
  defined by
\[ 
\hK
= \bigcap_{B \in \calK}  \fH_B.
\] 
  Thus, 
\[ 
 \hK(n) 
 =\big\{Y\in\smatng: \forall B\in \calK\, \exists T\succeq0
 \text{ such that }  \tr(T)\le 1 \text{ and } I\otimes T - \sum B_j \otimes Y_j \succeq 0\big\}.
\]

\begin{proposition}
\label{prop:dualofmatconset}
   If $\calK$ is matrix convex and each $\calK^\circ(n)$ is bounded $($equivalently,
   $\calK(1)$ contains $0$ in its interior$)$, then 
 \begin{enumerate}[\rm(i)]
\item 
$\displaystyle \hK(n) =\big\{Y\in\smatng: \exists T\succeq0, \,  \mbox{such that }  \tr(T)\le 1 \mbox{ and } \forall B\in \calK, \, I\otimes T - \sum B_j \otimes Y_j \succeq 0\big\}; $
\item 
$\displaystyle \hK(n) = \{SMS: M\in \calK^\circ(n), \ S\succeq 0, \ \tr(S^2)\le 1\}.$
\end{enumerate}
\end{proposition}

\begin{proof}
 Suppose $K$ is matrix convex.  To prove item (i), 
  let  $Y\in \hK(n)$ be given.
  For each $B$, let $\cT_B=\{T\in\cTn: I\otimes T - \sum B_j\otimes Y_j \succeq 0\}$.  Thus, the hypothesis that
 $Y\in\hK(n)$ is equivalent to assuming that for every $B$ in $\calK$, the set  $\cT_B$ is nonempty.

 That $\cT_B$ is compact will be verified by showing it satisfies the finite intersection property. 
 Now given $B^1,\dots, B^\ell\in \calK$, let $B=\bigoplus_k B^k\in \calK$. Since $B\in \calK,$   there is a $T$ such that
\[
 \bigoplus_k \big(I\otimes T - \sum B^k_j\otimes Y_j\big) = I\otimes T - \sum B_j \otimes Y_j \succeq 0. 
\]
 Hence $T\in \bigcap_{k=1}^\ell \cT_{B^k}$. It follows that the collection $\{\cT_B: B\in \calK\}$ has the finite intersection 
  property and hence there is a $T\in\bigcap_{B\in \calK} \cT_B$ and the forward inclusion in item (i) follows. 
  The reverse inclusion holds whether or not $\calK$ is matrix convex.

  To prove item (ii),  suppose $Y \in \hK(n)$. Thus, by what has already been proved, there is a positive semidefinite
  matrix $S$ such that  $\tr(S^2)\le 1$ and
\beq
\label{eq:hkmatconS}
I\otimes S^2 - \sum B_j \otimes Y_j   \succeq 0,
\eeq
 for all $B\in \calK$.  For positive integers $k$, let $S_k^+$ denote the inverse of $S+\frac{1}{k}$. Multiplying \eqref{eq:hkmatconS} on the left
 and on the right by $I\otimes S_k^+$ yields
\[
 I\otimes P -\sum B_j \otimes S_k^+ Y_j S_k^+ \succeq0,
\]
  where $P$ is the projection onto the range of $S$. 
 It follows that  $M_k=S_k^+YS_k^+ \in \calK^\circ(n)$.   Since
  $\calK^\circ(n)$ is bounded (by assumption) and closed, it is compact and
  consequently a subsequence of $(M_k)_k$ converges to some $M\in \calK^\circ(n)$. Hence, 
  $Y=SMS$.

  Reversing the argument above shows, if $M\in \calK^\circ(n)$ and $S$ is positive semidefinite with $\tr(S^2)\le 1$,
  then $Y=SMS\in \hK(n)$ and the proof is complete.
\end{proof}

\begin{proposition}
 \label{conj:tr-dual-LMI-dom}
   The \contra tracial dual $\hK$  of  a  free spectrahedron $\calK= \cD_{\mL_\Omega}$
   is exactly the set 
   $$\big\{\sum_\ell C_\ell^\ast \Omega C_\ell : \tr \big( \sum C_\ell^\ast C_\ell\big) \le  1\big\}.$$ 
\end{proposition}

\begin{proof}
 Suppose $Y$ is in the \contra tracial dual. 
 By Proposition \ref{prop:dualofmatconset}, 
  there is a positive semidefinite matrix $S$ with $\tr(S^2)\le 1$
   and an $M\in \calK^\circ$ such that $Y=SMS$. 
  Since $M\in \calK^\circ$,  by Remark \ref{rem:contractvaddzero} %
  there is a positive integer $\mu$ and a contraction $V$ such that
\[
  M = V^* (I_\mu \otimes \Omega)V = \sum_k^\mu  V_k^* \Omega V_k. 
\]
 Hence,
\[
   Y = \sum_k S V_k^* \Omega V_k S.
\]
 Finally,
\[
  \tr\big(\sum SV_k^* V_kS\big)\le \tr(S^2)\le 1.
\]

 Conversely suppose $\tr(\sum C_\ell^\ast C_\ell)\le 1$ and $Y=\sum C_\ell^\ast \Omega C_\ell$. 
  Let $T=\sum C_\ell^\ast C_\ell$ and note that for $B\in \calK$,
\[
  I\otimes T - \sum B_j\otimes Y_j = \sum_\ell C_\ell^\ast \big(I\otimes I - \sum_j B_j \otimes \Omega_j\big) C_\ell \succeq 0. \qedhere
\] 
\end{proof}

\subsubsection{\Idem Tracial Dual }
 \label{sec:Idem}
  Given a 
   free set  $\calK\subset\smatg$, we can define another dual
  set we call the \df{\idem} $ \calK^\tpd=(\calK^\tpd(m))_m$ 
  by
$$  \calK^\tpd(m)=   \{B\in\smatmg :  \calK \subset \fH_B \} 
$$
Equivalently, 
\[
  \calK^\tpd(m)=\big\{B\in\smatmg: \forall Y\in\calK \, \exists T\succeq0, \mbox{ such that }  \tr(T)\le 1 \mbox{ and } I\otimes T -\sum B_j\otimes Y_j \succeq 0 \big\}.
\]
 Each $\calK^\tpd(m)$ is levelwise convex. Moreover, if $B\in \calK^\tpd$ and $V^*V\preceq I$, then $V^*BV\in \calK^\tpd$.  On the other hand, there is no
 reason to expect that $\calK^\tpd$ is closed with respect to direct sums. Hence it need not be matrix convex.

  A subset $\cY$ of $\smatg$ is \df{contractively stable}
  if $\sum C_j^\ast Y C_j \in \cY$ for all $Y\in\cY$ such that  $\sum C_j^\ast C_j \preceq I.$ 
 In general, contractively stable sets need not be levelwise convex as Example \ref{ex:constablenotconvex} shows.

\begin{proposition}
 \label{prop:poDualcYisweakmatcon}
 The  set  $\calK^\tpd$ is
 contractively stable. 
 \end{proposition}

\begin{proof}
   Suppose $B\in\calK^\tpd(m)$.
    Let $n\times m$ matrices $C_1,\dots,C_\ell$ such that $\sum C_k^* C_k\preceq I$ be given
  and consider the $n\times n$ matrix $D=\sum C_k B C_k^*$.

 Given $Y\in\calK(p)$, there exists
 a positive semidefinite $p\times p$ matrix $T$ of trace at most one such that
\[
 I\otimes T -\sum B_j\otimes Y_j \succeq 0.
\]
 Thus,
\[
 \begin{split}
  I\otimes T - \sum_{j=1}^g D_j \otimes Y_j 
    &=  I\otimes T - \sum_{j=1}^g \sum_k C_k^* B_jC_k \otimes Y_j \\
    &=  (I-\sum C_k^* C_k)\otimes T + \sum_k (C_k\otimes I)^*\big(I\otimes T- \sum_j B_j\otimes Y_j\big)( C_k\otimes I) \succeq 0.
 \end{split}
\]
 Hence $D\in \calK^\tpd$ and the proof is complete. 
\end{proof}

   The \df{contractive convex hull} of $\cY$ is the smallest levelwise closed  set containing $\cY$ that is 
contractively stable. The following proposition finds the two
hulls defined by applying the two notions of tracial polar duals introduced above.

\begin{proposition}
 \label{prop:doubleup}
 For $\calK\subset \smatg$, the set $\widehat{(\calK^\tpd)}$
  is the levelwise closed convex contractively tracial hull of $\calK$. 
 Similarly, $\left(\widehat \calK\right)^\tpd$
  is the levelwise closed contractively stable  hull of $\calK$. 
\end{proposition} 

The proof of the second statement rests on the following companion
 to Lemma \ref{lem:separate}.  Recall, from equation \eqref{eq:fHT} the opp-tracial spectrahedron,
\[
 \fHT_Y =\{B: \exists T \succeq 0 \mbox{ such that } \tr(T)\le 1, \ \ I\otimes T - \sum B_j\otimes Y_j \succeq 0\}.
\]
 \index{opp-tracial spectrahedron}

\begin{lemma}
 \label{lem:separate-idem}
   Fix positive integers $m,n$, and  suppose that $\cS$ is a nonempty 
 subset of $\smatng$.  Let $\cU$ denote the subset of $\smatmg$  
  consisting of all tuples of the form
\[
 \sum_{\ell=1}^\mu C_\ell^* B^\ell C_\ell,
\]
 where each $C_\ell$ is $n\times m$, each $B^\ell\in \cS$ and 
  $\sum C_\ell^\ast C_\ell \preceq I$.  
  \ben[\rm(1)]
  \item
   If $Y\in\smatmg$ is in
  the conventional polar dual of $\cU$, then there exists 
  a positive semidefinite $m\times m$ matrix $T$ with trace at most one such that
\[
   I\otimes T -\sum B_j\otimes Y_j \succeq 0
\]
 for every $B\in\cS$. 
\item
 The tracial spectrahedra  $\fHT_Y$ are  closed and contractively stable. 
\item
 If $\calK\subset \smatg$ is contractively stable and if $Y\in\smatmg$ is in the conventional polar dual $\calK(m)\cpd$ of
 $\calK(m)$, then $\calK \subset \fHT_Y$.
\item
If $\calK\subset \smatg$ is levelwise  closed and convex, and contractively stable, then
\[
 \calK  = \bigcap \{ \fHT_Y : \fHT_Y\supset \calK\}= \bigcap_n \bigcap \{ \fHT_Y : Y\in\calK(n)\cpd \}.
\]
\item
The levelwise closed and convex  contractively stable hull of $\calK\subset\smatg$ is 
\[
  \bigcap \{ \fHT_Y : \fHT_Y \supset \calK\}.
\]
\item
 For $\calK\subset \smatg$, we have $Y\in\hK(n)$ if and only if $\calK\subset \fHT_Y$. 
\een
\end{lemma}

\begin{proof}
 The proof of item (1) is similar to the proof of Lemma \ref{lem:separate}  and is omitted.
 Likewise, the proof of item (2) follows an argument given in the proof of  Proposition \ref{prop:poDualcYisweakmatcon}.

 To prove (3), suppose that $Y\in\calK(m)\cpd$.
   Given $B\in\calK(n)$, an application
 of the first part of the lemma with $\cS=\{B\}$ produces an $m\times m$ positive semidefinite
 matrix $T$ with $\tr(T)\le 1$ such that   $I\otimes T -\sum B_j\otimes Y_j \succeq 0$. 
 Hence, $B\in \fHT_Y.$
 
 Moving on to item (4). From (3), if $Y\in\calK(m)\cpd$, then $\calK \subset \fHT_Y$.
 On the other hand, if $Y\in\smatmg$ and $\calK\subset \fHT_Y$, then, 
 for $B\in\calK(m)$,
\[
  I\otimes T - \sum B_j\otimes Y_j \succeq 0
\]
 for some positive semidefinite $T$ with trace at most one. In particular, 
 with $e=\sum_{s=1}^m  e_s\otimes e_s\in\mathbb R^m\otimes \mathbb R^m$,
\[
  0 \le  \langle I\otimes T - \sum B_j\otimes Y_j e,e\rangle  = \tr(T) - \lambda_Y(B).
\]
 Hence $Y\in\calK(m)\cpd$.  Continuing with the proof of (4), 
  from item (3), 
\[
 \calK \subset \bigcap_n \bigcap\{\fHT_Y: Y\in\calK(n)\cpd \}.
\]
 To establish the reverse inclusion, suppose that $C$ is not in 
  $\calK(m)$. Since  $\calK(m)$ is assumed closed and convex,
  there exists a $Y\in\calK(m)\cpd$, the conventional polar dual of $\calK(m)$ (so that
 $\lambda_Y(\calK(m))\le 1$) with $\lambda_Y(C)>1$.  In particular, $\calK\subset \fHT_Y$.
 On the other hand, if $T$ is $m\times m$ and  positive semidefinite with trace at most one, then
  with $e=\sum e_s\otimes e_s,$ 
\[ 
 \langle (I\otimes T-\sum C_j\otimes Y_j)e,e \rangle
  =\tr(T) - \sum_j \tr(C_jY_j) = \tr(T)-\lambda_Y(C) <0.
\]
 Hence, $C\notin \fHT_Y$. 
 
 To prove item (5), let $\mathcal H$ denote the contractively stable hull of $\calK$. 
 Let also $\mathcal I$ denote the intersection of tracial spectrahedra $\fHT_Y$
  such that $\calK\subset \fHT_Y$.  Evidently $\mathcal H \subset\mathcal I$.  On
  the other hand, using item (4), 
\[
 \mathcal K \subset \mathcal H  \subset \mathcal I \subset \bigcap \{\fHT_Y: \fHT_Y  \supset \mathcal H \} =\mathcal H.
\]

 Finally, for item (6), first suppose $Y\in\hK(n)$. By definition, for each $B\in\calK$
 there is a positive semidefinite $T$ of trace at most one such that 
 $I\otimes T-\sum B_j\otimes Y_j \succeq 0$.  Hence $\calK \subset \fHT_Y$.
 Conversely, if $B\in\fHT_Y$, then $I\otimes T-\sum B_j\otimes Y_j\succeq 0$
  for some positive semidefinite $T$ of trace at most one depending on $B$. 
  Thus, if $\calK\subset \fHT_Y$, then $Y\in\hK(n)$. 
\end{proof}

\begin{proof}[Proof of Proposition {\rm\ref{prop:doubleup}}]
 Since
$$
\widehat{(\calK^\tpd)} = \bigcap_{B \in \calK^\tpd}  \fH_B =
\bigcap_{ \calK \subset \fH_B }  \fH_B, 
$$
item (ii) of Theorem \ref{thm:tracehull}  gives the conclusion of the first
 part of the proposition.

 Likewise, 
\[
 (\hat{\calK})^\tpd = \{B: \hat{\calK} \subset \fH_B\}
  = \bigcap_{Y\in\hat{\calK}} \fHT_Y =\bigcap\{\fHT_Y: \fHT_Y\supset \calK\}
\]
 and the term on the right hand side is, by Lemma \ref{lem:separate-idem},
  the closed contractive convex  hull  of $\calK$. 
\end{proof}

\subsection{Matrix Convex Tracial Sets and Free Cones}
 \label{sec:tcs}
 In this subsection we  introduce and study properties of free (convex) cones.

 A subset $\cS$ of $\bbS^g$ is a \df{free cone}  
 if for all positive integers $m,n,\ell$, tuples $T\in\cS(n)$ and $n\times m$ matrices
 $C_1,\dots,C_\ell$, the tuple $\sum C_i^*TC_i$ is in $\cS(m)$. The set $\cS$ is a 
 \df{free convex cone} if for all positive integers $m,n,\ell$,
  tuples $T^1,\dots,T^\ell \in \cS(n)$ and $n\times m$ matrices $C_1,\dots,C_\ell$,
  the tuple $\sum_i C_i^* T^i C_i$ lies in $\cS(m).$ 
  Finally, a subset $\cY$ of $\smatg$ is a \df{contractively tracial convex set} if $\cY$ is contractively tracial
  and given positive integers $m$, $n$, $\mu$ and $Y^1,\dots,Y^\mu \in \cY(m)$ and 
  $n\times m$ matrices $C_1,\dots,C_\mu$ with $\sum C_j^\ast C_j \preceq I$, the tuple
\[
  \sum C_j Y^j C_j^\ast
\]
  lies in $\cY(n)$. 
  This condition is an analog to matrix convexity
  of a set containing $0$ which we studied earlier in this paper.
Surprisingly:

  \begin{prop}\label{prop:surpCone}
    Every contractively tracial convex set is a free convex cone.
  \end{prop}
    
  For the proof of this proposition we introduce an auxiliary notion and then give a lemma.  
A subset $\cY$ of $\smatg$
 is \df{closed with respect to identical direct sums}
   if for each $Y\in\cY$ and positive integer $\ell$, the tuple $ I_\ell\otimes Y$
   is in $\cY$.

\begin{lemma}
 \label{lem:sums-convex-redo}
 Suppose $\cY\subseteq\mathbb S^g$.
 \ben[\rm (1)]
 \item If $\cY$ 
  is contractively tracial and  closed with respect to identical direct sums, 
  then $\cY$ is a free  cone. 
\item
  If $\cY$ is contractively tracial and closed with respect to direct sums, then $\cY$
  is a free convex cone. 
\item
 If  $\cY$ is a tracial set containing $0$ which is levelwise convex and closed with respect to identical direct sums, 
 then each $\cY(m)$ is a  cone in the ordinary sense.
 \een
\end{lemma}

\begin{proof} 
 To prove the first statement,  let  $Y\in \cY(n)$ and a positive integer $\ell$ be given.
 Let $V_k$  denote the block $1\times \ell$ row matrices with $m\times n$
  matrix entries with $I_n$ in the $k$-th position and $0$ elsewhere,
  for $k=1,\dots,\ell$.  It follows that $\sum V_k^\ast V_k = I$.
  Since also $I_\ell \otimes Y$ is in $\cY$ and $\cY$ is tracial,
\[
  \sum V_k \big(Y\otimes I_\ell \big) V_k^* = kY \in \cY(n).
\]

  Now let  positive integers $m$ and $\ell$ 
  and $m\times n$ matrices $C_1,\dots,C_\ell$ 
 and $Y^1,\dots,Y^\ell \in \cY(n)$ be given. 
  Choose a positive integer $k$ such that each
  $D_j=\frac{C_j}{\sqrt{k}}$ has norm at most one. 
  Consider $M_j$ equal the block $1\times \ell$ row matrix  with
   $m \times n$ entries  with $D_j$  in the $j$-th position
 and $0$ elsewhere, for $j=1,\dots,\ell$.  It follows that 
\[ 
\sum_j M_j^\ast M_j =  \mbox{diag}( D_1^* D_1 , \dots , D_\ell^* D_\ell ) \preceq I.
\] 
 Since $\cY$ is tracial, and assuming either $Y^j=Y^k$ for all $j,k$ 
  and $\cY$ is closed under identical direct sums 
  or assuming that $\cY$ is closed under direct sums, $\oplus_{j=1}^\ell  Y^j$ is in $\cY$ and hence,
\[ 
  \sum M_j (\oplus^\ell Y^j ) M_j^\ast  =  k \sum_j^\ell D_j Y^j D_j^*  =\sum C_j Y^y C_j^* \in\cY(n).
\] 
  Thus, in the first case $\cY$ is a free cone and in the second a free convex cone.

To prove the third statement, note that 
the argument used to prove the first part of the lemma shows,
  if $\cY$ is a tracial set that is closed with respect
 to identical direct sums and if each $C_j=I$, then
  $\ell Y = \sum C_j (Y\otimes I_\ell) C_j^\ast$ is in 
  $\cY(n)$.  If $\cY$ is levelwise convex,
 since also $0\in\cY(n)$, it follows that $\cY(n)$ is a convex cone. 
\end{proof}

\begin{proof}[Proof of Proposition {\rm\ref{prop:surpCone}}]
  Fix positive integers $n$ and $\nu$. Let $Y^1,\dots,Y^\nu \in \cY(n)$ be given.  Let
  $C_\ell$ denote the inclusion of $\mathbb  R^n$ as the $\ell$-th coordinate in
   $\mathbb R^{n\nu} = \bigoplus_{i=1}^\nu \mathbb R^n$.  In particular, $C_\ell^\ast C_\ell = I_n$
  and hence, $Z^\ell = C_\ell Y^\ell C_\ell^\ast \in \cY(n\nu)$ (based only on $\cY$ being a tracial set).
   Now let $V_\ell$ denote the block $\nu\times \nu$ matrix with $n\times n$ entries with $I_n$ in the $(\ell,\ell)$ position
   and zeros ($n\times n$ matrices) elsewhere.  Note that $\sum V_\ell^\ast V_\ell =I_{n\nu}$. Hence, 
\[
 \sum V_\ell Z^\ell V_\ell^* = \mbox{diag}\begin{pmatrix} Y^1 & Y^2 &\dots & Y^\nu \end{pmatrix} \in \cY(n\nu).
\]
 Thus $\cY$ is closed with to identical direct sums.
 By the second part of Lemma \ref{lem:sums-convex-redo}, 
  $\cY$ is a free convex cone.
\end{proof}

 \begin{rem}\rm
 \label{lem:tonowhere}
 If $\cY\subset\smatg$ is a  cone 
 and if $B\in\smatng$ is in the polar dual of the set $\cU$ consisting of tuples
 $\sum C_j Y^j C_j^*$  for $Y^j\in\cY$ and $C_j$ such that $\sum C_j^* C_j \preceq I$, then 
\[
   \sum B_j\otimes Y_j \preceq 0
\]
 for all $Y\in \cY(m)$.   
 In particular,
 the  polar   dual $\cB=\cY^\circ$ of a  cone $\cY$ is a free convex cone.

  To prove this assertion, pick $B\in\smatng$ in the polar dual of $\cU$. Fix a positive integer $m$. 
By Lemma \ref{lem:separate},  there exists a positive semidefinite $T$ with trace at most one such that
\[
 I\otimes T - \sum B_j\otimes Y_j \succeq 0
\]
 for all $Y\in\cY(m)$.  Since $\cY(m)$ is a cone,
  $I \otimes T - \sum   B_j  \otimes t^2  Y_j  \succeq 0$
for all real $t$ and hence
$$
  - \sum B_j \otimes Y_j\succeq 0. 
$$
It follows that
\beq
\label{eq:coneCBC}
  - \sum C^* B_j C  \otimes Y_j\succeq 0
\eeq
for any $C$. 
The conventional polar dual of a set is  convex,
which implies convex combinations with various $C_j$ in \eqref{eq:coneCBC}
are in $\cB$.
 Hence $\cB$ is a free convex cone.  
\end{rem}

\section{Examples}
\label{sec:exs}
  The examples referenced in the body of the paper are gathered together in this section.
  Some of the examples consider the scalar level $\Gamma(1)\subset\R^g $ of a   free set $\Gamma\subset\mbS^g$.

\begin{example}\rm
 \label{ex:fails}
This example  shows that it is not necessarily possible to choose $V$ an isometry in equation \eqref{eq:LMIdominate2} of Theorem \ref{thm:convexPos} if the boundedness assumption on $\cD_{\mL_B}$ is omitted.  Let $g=1$, and consider $\mL_A(x)=1+x$, $\mL_B(x)=1+2x$. In this case,
\[
\cD_{\mL_B}=\Big\{ X: X\succeq-\frac12\Big\} \subseteq \cD_{\mL_A}=
\{X: X\succeq-1\} .
\]
 It is clear that there does not exists a $\mu$ and isometry $V$ such that $A = V^*(I_\mu\otimes B)V$.  This example is in fact representative in the sense that  if 
$\mL_B$ is a monic linear pencil and $\cD_{\mL_B}$  is unbounded, then  there is a monic linear pencil $\mL_A$ with $\cD_{\mL_B}\subseteq\cD_{\mL_A}$ for which there does not exist a $\mu$ and isometry $V$ such that $A = V^*(I_\mu\otimes B)V$.
\end{example}

\begin{example} \label{ex:no tracial extension}
Here is an example of a 
trace preserving cp map $\phi:\cS\to M_2$ with domain  an operator system $\cS$ that 
does not admit an extension to a trace non-increasing 
 cp map $\phi:M_2\to M_2$. This phenomenon contrasts with the
 classical Arveson extension theorem \cite{Arv69} which says that
 any ucp map extends to the full algebra.

Let $\cS=\Span \{I_2,E_{1,2},E_{2,1}\}$, 
\[
V=\bem \sqrt{\frac12} & 0 \\ 0 & \sqrt{\frac32}\eem,
\]
and
consider the cp map $\phi:\cS\to M_2$,
\[
\phi(A)=V^* AV\quad\text{for }A\in\cS.
\]
We have
\[
\phi(I_2)= V^* V =\bem \frac12 & 0 \\ 0 &\frac32\eem, \quad
\phi(E_{1,2})=\frac{\sqrt3}2 E_{1,2}, \quad
\phi(E_{2,1})=\frac{\sqrt3}2 E_{2,1},
\]
so $\phi$ is trace preserving on $\cS$.

Now let us consider a  cp extension (still denoted by $\phi$) 
of $\phi$ to $M_2$.
Letting
\[
\phi(E_{1,1})=\bem a & b \\ b & c\eem,
\]
the Choi matrix for $\phi$ is
\[
C=\bem
a & b & 0 & \frac{\sqrt3}2 \\[.1cm]
b & c & 0 & 0 \\[.1cm]
0 & 0 & \frac12-a & -b\\[.1cm]
 \frac{\sqrt3}2 & 0 & -b & \frac32-c
 \eem\succeq0.
 \]
Supposing $\phi:M_2\to M_2 $ is trace non-increasing, 
\[
\begin{split}
1&=\tr(E_{1,1}) \geq \tr( \phi(E_{1,1})) = a+c \\
1&=\tr(E_{2,2}) \geq \tr( \phi(E_{2,2})) = 2-a-c,
\end{split}
\]
whence $a+c=1$.
Since $C$ is positive semidefinite, the nonnegativity of
the diagonal of $C$ now gives us
\[
0\leq a\leq\frac12.
\]
But then the  $2\times2$ minor 
\[
\bem a & \frac{\sqrt3}2\\[.1cm]
\frac{\sqrt3}2 & \frac12+a
\eem
\]
is not positive semidefinite, a contradiction.
\end{example}

\begin{example}\rm
\label{ex:btv}\label{ex:btvv}\label{ex:btvlift}
Consider 
\[ 
p=1-x_1^2-x_2^4.
\] 
In this case $p$ is 
 symmetric with $p(0)=1>0$.  
 
 \begin{center}
\includegraphics[scale=1.2]{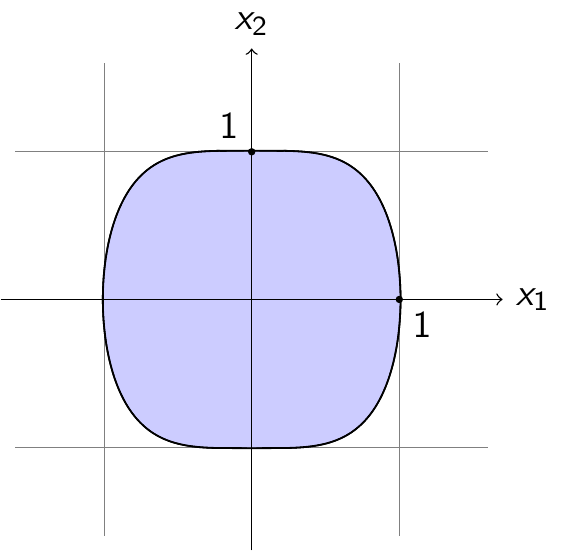}
\end{center}~\centerline{Bent TV screen $\cD_p(1)=\{ (x_1,x_2)\in\R^2 : 1-x_1^2-x_2^4\geq0 \}$ }

The free semialgebraic set $\cD_p$ is called the
real \df{bent free TV screen}, or (bent) TV screen for short.
While $\cD_p(1)$ is convex, 
it is known that $\cD_p$ is not 
matrix convex, see 
\cite{DHM07}
or \cite[Chapter 8]{BPR13}. Indeed, already $\cD_p(2)$ 
is not a convex set.

\begin{center}
\includegraphics[scale=.35]{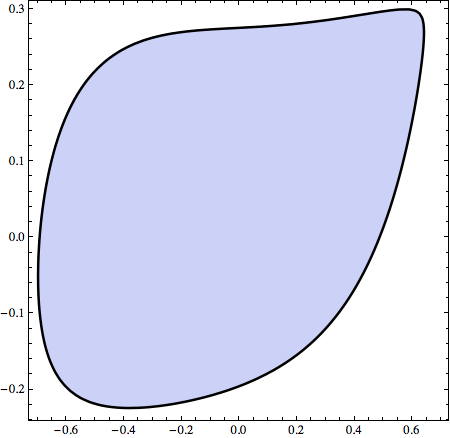}
\end{center}~\centerline{A non-convex 2-dimensional slice of $\cD_p(2)$.}

  That the set $\cD_p(1)$ is a spectrahedral shadow is well known. Indeed, letting
\[
  L(x_1,x_2,y) = \begin{pmatrix} 1  & 0 & x_1 \\ 0 & 1 & y \\ x_1 & y & 1 \end{pmatrix}
    \oplus \begin{pmatrix} 1 & x_2 \\ x_2 & y\end{pmatrix},
\]
 it is readily checked that $\proj_x \cD_{L}(1) = \cD_p(1)$.  Further,
  Lemma \ref{lem:boundedmonicLift} implies that $L$ can 
  be replaced by a monic linear pencil $\mL$. An explicit 
  construction of such an $\mL$ can be found in \cite[\S7.1]{Lasse}.
  We remark that $\proj_x \cD_L$ strictly contains the 
  matrix convex hull of 
  $\cD_p$, cf.~\cite[\S7.1]{Lasse}.
\end{example}

The next example is one in a classical commutative situation. We refer the reader to  \cite{BPR13} for background on classical convex algebraic geometry.

\begin{example}
\label{ex:scalarbentTVpd}
The polar dual of 
 the \bTV{}
$\cD_p=\{(X,Y): 1-X^2-Y^4\succeq0\}$. 
We note that $\cD_p^\circ(1)$  
 coincides with the classical polar dual of $\cD_p(1)$ by  Proposition
\ref{prop:poDual}, cf.~\cite[Example 4.7]{Lasse}.

We first find the boundary $\pt\cD_p^\circ(1)$ using
 Lagrange multipliers.  Consider a linear function $1-(c_1 x+ c_2 y)$
 that is nonnegative but not strictly positive on $\cD_p^\circ(1)$
and its values on (the boundary of) $\cD_p(1)$. The  Karush–Kuhn–Tucker (KKT)  conditions for first order optimality give 
\[
1-x^2-y^4=0 ,\quad c_1=2\la x, \quad c_2 = 4\la y^3, \quad 1=c_1x+c_2y.
\]
Eliminating $x,y,\la$ leads to the following formula relating $c_1,c_2$:
\[
q(c_1,c_2):=-16 c_1^8+48  c_1^6-48 c_1^4-8 c_1^4
   c_2^4+16 c_1^2-20  c_1^2 c_2^4
   -c_2^8+ c_2^4=0.
   \]
Thus
the boundary of $\cD_p^\circ(1)$ is contained in the  zero set of
$q$. 
Since $q$ is irreducible,
 $\pt \cD_p^\circ(1)$  in fact equals the zero set of $q$.
In particular, $\cD_p^\circ(1)=\{(x,y)\in\R^2: q(x,y)\geq0
\}$ is not a
spectrahedron, since it fails the line test in \cite{HV07}.


\begin{center}
\includegraphics{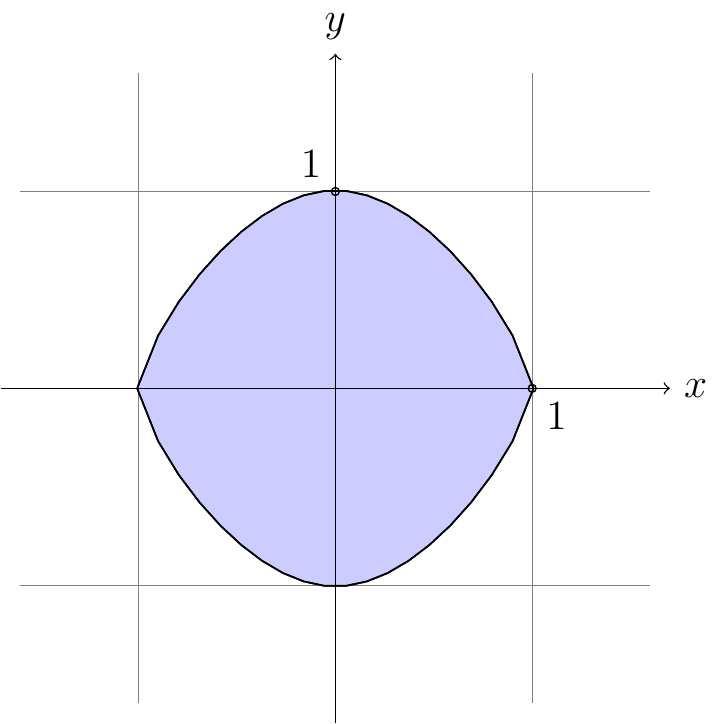}
\end{center}
\centerline{Polar dual $\cD_p^\circ(1)$ of the bent TV screen.}
\vspace{-.6cm}
\end{example}

\begin{example}\rm
  Recall the free bent TV screen is the nonnegativity set $\cD_p$ for the polynomial
 $p=1-x^2-y^4$ (see Example \ref{ex:btv}).  Let $\calK$ denote the closed matrix convex hull of $\cD_p$.  
Then $\calK(1)=\cD_p(1)$ and hence, by 
  Proposition \ref{prop:poDual} and Example \ref{ex:scalarbentTVpd},
  $\cD_p^\circ(1) = \cD_p(1)^\circ$ is not a spectrahedron. Hence,
  $\cD_p^\circ$ is not a free spectrahedron. In particular,
$\calK$ 
cannot be represented by a single $\Omega$ as
  in Theorem \ref{thm:polarLMI}.
\end{example}

\begin{example}
\label{ex:nocon}
\rm
 Tracial and contractively tracial hulls need not be convex (levelwise) as this example shows. 
 Consider
\[
 A= \begin{pmatrix} 1 & 0 \\ 0  & 0\end{pmatrix}, \ \ \ B=\begin{pmatrix} 0 & 0 \\ 0 & -1 \end{pmatrix}.
\]
 To show that $D=\frac12(A+B)$ is not in $\thull(\{A,B\})$, suppose there exists
 $2\times 2$ matrices $V_1,\dots,V_m$ such that $\sum V_j^* V_j =I$ and 
\[
 \sum V_j A V_j^* = D.
\]
 On the one hand, the trace of $D$ is zero, on the other hand, $\sum V_jA V_j^*$ has trace $1$. Hence $D$ is not in the tracial hull of $A$.
 A similar argument shows that $D$ is not in the tracial hull of $B$. Hence by Lemma \ref{lem:hullofS},
 $D\not\in\thull(\{A,B\})$. 

 Now consider the tuples $A=(A_1,A_2)$ and $B=(B_1,B_2)$ defined by,
\[
  A_1 = \begin{pmatrix} 1 & 0\\0 & 0\end{pmatrix}=-B_2, \ \ \  A_2=\begin{pmatrix} 0 & 0\\ 0 & 1 \end{pmatrix} = -B_1.
\]
 In this case $D=\frac12(A+B)$ is,
\[
 D=(D_1,D_2) = \frac12 \begin{pmatrix} \begin{pmatrix} 1&0\\0 & -1\end{pmatrix}, \begin{pmatrix} -1&0\\0&1\end{pmatrix} \end{pmatrix}. 
\]
 Suppose $\sum C_j^* C_j\preceq I.$ Let
\[
 F_k = \sum C_j A_k C_j^*
\]
 and note  $\tr(F_k)\ge 0.$ On the other hand, $\tr(D_k)=0$.  Hence, if $F_k=D_k$, then 
\[
 0 = \tr(F_k) = \sum_j \tr(C_j A_k C_j^*)\ge 0.
\]
 But then, for each $j$,
\[
 0  = \tr\big((C_j (A_1+A_2)C_j^*\big) =\tr(C_jC_j^*).
\]
 It follows that $C_j=0$ for each $j$ and thus $F_k=0$, a contradiction. Thus, $D$ is not in the contractive tracial hull of $A$ and by symmetry
 it is not in the contractive tracial hull of $B$.  
By Lemma \ref{lem:cthullofS}, $D$ is not in the contractive tracial hull generated
 by $\{A,B\}$. 
\end{example}

The following example shows a contractively stable set need not be convex.
 \begin{example}
\label{ex:constablenotconvex}
Consider the $2\times 2$ matrices $A,B$ from Example \ref{ex:nocon}.
The smallest contractively stable set containing $A,B$ is
the levelwise closed set
\[
\cY=\{ \sum C_j^*AC_j : \sum C_j^*C_j\preceq I \}
\cup
\{ \sum D_j^*BD_j : \sum D_j^*D_j\preceq I \}.
\]
Each matrix in $\cY$ is either positive semidefinite or negative
semidefinite, so $\frac12(A+B)\not\in\cY$.
 \end{example}

\linespread{1.0}

\linespread{1.15}

\newpage

\printindex

\newpage
\normalsize

\tableofcontents

\end{document}